\documentclass[a4paper,11pt]{article}
\usepackage[left=2cm,top=2cm,right=2cm,bottom=2cm,bindingoffset=0.5cm]{geometry}

\usepackage{amsfonts,graphicx,caption,subcaption,amsmath,amssymb,enumitem,
,url,authblk,amsthm,multicol,cite,float,bigints,tabularx,slashbox,bigints}
\usepackage[rgb,dvipsnames]{xcolor}

\DeclareMathOperator\erf{erf}

\theoremstyle{definition}
\newtheorem{definition}{Definition}[section]
\theoremstyle{theorem}
\newtheorem{theorem}{Theorem}[section]
\theoremstyle{lemma}
\newtheorem{lemma}{Lemma}[section]
\theoremstyle{corollary}

\theoremstyle{theorem}
\newtheorem{remark}{Remark}
\theoremstyle{theorem}

\title{An exact solution to a Stefan problem with variable thermal conductivity and a Robin boundary condition}

\author[1,2]{Andrea N. Ceretani\footnote{{\sc Corresponding author}:\\ \hspace*{0.6cm}Andrea N. Ceretani. E-mail: aceretani@austral.edu.ar}}
\author[3,4]{Natalia N. Salva}
\author[1]{Domingo A. Tarzia}

\affil[1]{{\footnotesize CONICET - Depto. de Matem\'atica, Facultad de Ciencias Empresariales, Univ. Austral, Paraguay 1950, S2000FZF Rosario, Argentina.}}
\affil[2]{{\footnotesize Depto. de Matem\'atica, Facultad de Ciencias Exactas, Ingenier\'ia y Agrimensura, Univ. Nacional de Rosario, Pellegrini 250, S2000BTP Rosario, Argentina.}}
\affil[3]{{\footnotesize CONICET - CNEA, Depto. de Mec\'anica Computacional, Centro At\'omico Bariloche, Av. Bustillo 9500, 8400 Bariloche, Argentina.}}
\affil[4]{{\footnotesize Depto. de Matem\'atica, Centro Regional Bariloche, Univ. Nacional del Comahue, Quintral 250, 8400 Bariloche, Argentina.}}

\date{}

\begin{document}  
\maketitle

\begin{abstract}
In this article it is proved the existence of similarity solutions for a one-phase Stefan problem with temperature-dependent thermal conductivity and a Robin condition at the fixed face. The temperature distribution is obtained through a generalized modified error function which is defined as the solution to a nonlinear ordinary differential problem of second order. It is proved that the latter has a unique non-negative bounded analytic solution when the parameter on which it depends assumes small positive values. Moreover, it is shown that the generalized modified error function is concave and increasing, and explicit approximations are proposed for it. Relation between the Stefan problem considered in this article with those with either constant thermal conductivity or a temperature boundary condition is also analysed.  
\end{abstract}

{\bf Keywords}: Stefan problems, exact solutions, temperature dependent thermal conductivity, convective boundary conditions, modified error function, phase-change processes. 

\section{Introduction}\label{Sect:Introduction}
The understanding of phase-change processes has been inspiring scientists from the earlier 18th century. Already in 1831, Lam\'e and Clapeyron studied problems related to the solidification of the Earth planet \cite{LaCl1831}. Also the mathematical formulation of phase-change processes as {\em free boundary problems} dates from the 18th century, since it owes much to the ideas developed by Stefan in 1889 \cite{St1889-1,St1889-2,St1889-3}. At present, their study is still an active area of research. Besides phase-change process are interesting in themselves, they attract interests because they are present in a wide variety of situations, both natural and industrial ones. Glass manufacturing and continuous casting of metals are examples of industrial activities involving them, some recent works in this area are \cite{BoIv2014,GaToTu2009}. Controlling side-effects of certain industrial processes or preventing future problems derived from our energy-dependent lifestyle, are also examples of how phase-change 
processes arise as a subject of study \cite{FuFaPr2014, ElKh2013}. Permafrost phenomena or dynamics of snow avalanches are examples of natural situations whose study involves phase-change processes, some recent articles in these subjects are \cite{KuHa2016, AkMoSt2014,CaFuFa2016}. We refer the reader to \cite{ChShVa2015,Ta2011} and the references therein for a recent survey in applications and future challenges in free boundary problems. Other references can be seen in the last published Free Boundary Problems International Conference Proceedings \cite{FiRoSa2007}.      

In this article we will focus on phase-change processes that are ensued from an external temperature imposed at some part of the fixed boundary of a homogeneous material. A classical simplification in modelling this sort of phenomena is to consider boundary conditions of Dirichlet type (temperature conditions). This is based on the assumption that heat is instantaneously transferred from the external advise through which a specific temperature is imposed to the material. In view that is physically unrealistic, several authors have suggested to consider conditions of Robin type (convective conditions) since they mimic the fact that the heat transfer at the boundary is proportional to the difference between the imposed temperature and the one the material presents at its boundary (see for example the books \cite{AlSo1993,CaJa1959}). Another classical simplification when modelling phase-change processes is to consider that thermophysical properties are constant. Though it is reasonable for most phenomena under 
moderate temperature variations \cite{AlSo1993}, it is not what actually happens as a rule. In fact, this hypothesis has been removed in many works in the attempt to improve the mathematical model (see, for example \cite{CeTa2014,BrNa2015,SaTa2011-b,Vo2004}). All this have encouraged us to look at phase-change processes with convective boundary conditions and non-constant physical properties. 

In 1974, Cho and Sunderland studied a phase-change process for a one-dimensional semi-infinite material with temperature-dependent thermal conductivity \cite{ChSu1974}. The dependence was assumed to be linear, which  is a quite good approximation of what actually happens with several materials (water, for example \cite{AlSo1993}). The phase-change process was assumed to be ensued from a constant temperature imposed at the fixed boundary of the body, what was modelled through a Dirirchlet condition. For the resulting Stefan problem, Cho and Sunderland have presented an exact similarity solution. The temperature was obtained through an auxiliary function $\Phi$ that they have called a {\em Modified Error} (ME) function and that was defined as the solution to a nonlinear ordinary differential problem of second order. Revisiting the work of Cho and Sunderland, a couple of curiosities have arised. On one hand, the existence of the ME function was not proved there. Despite of this lack of theoretical results, 
the ME function was widely used in the context of phase-change processes before their existence and uniqueness were proved in the recent article \cite{CeSaTa2017} (see, for example, \cite{Cr1956,Wa1950,BrNaTa2007,FrVi1987,Lu1991,OlSu1987,SaTa2011-b,SaTa2011-c,Ta1998,CoKa1994}). On the other hand, by following the arguments presented in \cite{ChSu1974} it is obtained that the ME function must satisfy a differential problem over a closed bounded interval $[0,\lambda]$ with $\Phi(0)=0$, $\Phi(\lambda)=1$. Nevertheless, in \cite{ChSu1974} it was considered a boundary value problem over $[0,+\infty)$ with $\Phi(0)=0$, $\Phi(+\infty)=1$. Although in this way it is clearer the relation between the modified and classical error functions (see \cite{ChSu1974,CeSaTa2017} for further details),
%\begin{equation}\label{erf}
%\erf(x)=\frac{2}{\sqrt{\pi}}\displaystyle\int_0^x\exp(-z^2)dz,\quad x\geq 0,
%\end{equation}
the change made by Cho and Sunderland add some extra conditions on the temperature function.  

In this article we consider a similar phase-change process to that studied in \cite{ChSu1974}. We are mainly motivated by: a) improving the modelling of the imposed temperature at the fixed boundary by considering a convective boundary condition, b) obtaining a solution of similarity type without any extra condition on the temperature distribution. We will study a solidification process, but a completely similar analysis can be done for the case of melting. Aiming for simplicity, we will restrict our presentation to a one-phase process. That is, the case in which the material is initially liquid at its freezing temperature.

The organisation of the paper is as follows. First (Sect. \ref{Sect:characterisation}), we introduce the one-phase Stefan problem through which we will study the phase-change process. In this Section we also present a characterisation for any similarity solution to the Stefan problem in terms of a {\em Generalized Modified Error} (GME) function. This will be defined as the solution to a nonlinear boundary value problem of second order. Similarly to \cite{ChSu1974}, this problem will depend on a positive parameter $\beta$ related to the slope of the thermal conductivity as a linear function of the temperature distribution. Next (Sect. \ref{Sect:EyU}), we analyse the existence of the similarity solutions given in Section \ref{Sect:characterisation}. In particular, we prove that there exists a unique non-negative analytic GME function when $\beta$ assumes small positive values. We also prove that this GME function is concave and increasing, and explicit approximations are proposed for it. Finally (Sect. \ref{Sect:Comparison}), we discuss about how the Stefan problem presented in Section \ref{Sect:characterisation} is related to those studied in \cite{ChSu1974} (Dirichlet condition at the fixed boundary) and \cite{Ta2017} (constant thermal conductivity). 

\section{The Stefan problem.}\label{Sect:characterisation}

The one-phase solidification process introduced in Section \ref{Sect:Introduction} will be studied through the following Stefan problem:

\begin{subequations}\label{StefanProblem}
\begin{align}
\label{Eq:heat}&\rho c T_t(x,t)=(k(T(x,t))T_x(x,t))_x&0<x<s(t),\,t>0\\
\label{Cond:s0}&s(0)=0&\\
\label{Cond:FreeBoundaryTemp}&T(s(t),t)=T_f&t>0\\
\label{Cond:Stefan}&k(T_f)T_x(s(t),t)=\rho l\dot{s}(t)&t>0\\
\label{Cond:Conv}&k(T(0,t))T_x(0,t)=\frac{h_0}{\sqrt{t}}(T(0,t)-T_\infty)&t>0
\end{align}
\end{subequations}

\noindent In (\ref{StefanProblem}), the unknown functions are the temperature $T$ of the solid region and the free boundary $s$ separating the phases. The parameters $\rho>0$ (density), $c>0$ (specific heat), $l>0$ (latent heat per unit mass), $h_0>0$ (coefficient related to the heat transfer at $x=0$), $T_f\in\mathbb{R}$ (freezing temperature) and $T_\infty<T_f$ (constant temperature imposed in the neighbourhood of the boundary $x=0$) are all known constants. The function $k$ (thermal conductivity) is defined as:
\begin{equation}\label{K}
k(T)=k_0\left(1+\beta\frac{T-T_\infty}{T_f-T_\infty}\right),
\end{equation}  
where $k_0>0$, $\beta>0$ are given constants.

\begin{remark}\label{Re:Conv-Dirichlet}
Let us assume for a moment  that for each $h_0>0$ exists a solution to problem (\ref{StefanProblem}) such that $T(0,\cdot)$, $T_x(0,\cdot)$ admit bounds independent of $h_0$ (what actually happens in the most common physical situations). Then, by taking the limit when $h_0\to\infty$ for each fixed $t>0$ in (\ref{Cond:Conv}), we obtain:
\begin{equation}\label{Cond:Dirichlet}
T(0,t)=T_\infty\quad t>0.\tag{\ref{Cond:Conv}$^\dag$}
\end{equation}     
In other words, if we were able to consider an infinite value for the heat transfer coefficient $\frac{h_0}{\sqrt{t}}$ in the convective boundary condition (\ref{Cond:Conv}), then the temperature function given through problem (\ref{StefanProblem}) would satisfy the temperature boundary condition (\ref{Cond:Dirichlet}). Thus, the mathematical framework given by problem (\ref{StefanProblem}) agrees well with the physical ideas about temperature and convective boundary conditions discussed in Section \ref{Sect:Introduction} (see, for example, \cite{AlSo1993,CaJa1959} for a detailed explanation of physical interpretations of boundary conditions).       
\end{remark}
  
We are interested here in obtaining a similarity solution to problem (\ref{StefanProblem}). More precisely, one in which the temperature $T(x,t)$ can be written as a function of the single variable $\frac{x}{2\sqrt{\alpha_0t}}$, where $\alpha_0=\frac{k_0}{\rho c}>0$ (thermal diffusivity for $k_0$). Through the following change of variables in problem (\ref{StefanProblem}):
\begin{equation}\label{DimensionlessT}
\varphi\left(\frac{x}{2\sqrt{\alpha_0t}}\right)=\frac{T(x,t)-T_\infty}{T_f-T_\infty}\quad 0<x<s(t),\,t>0,
\end{equation}
and a few simple computations, the following theorem can be proved (we refer the reader to the proof of Theorem 2 in \cite{Ta2004} for an illustrative example of this sort of demonstrations).

\begin{theorem}\label{Th:characterisation}
The Stefan problem (\ref{StefanProblem}) has the similarity solution $T$, $s$ given by:
\begin{subequations}\label{SimSol}
\begin{align}
\label{Sol:T}&T(x,t)=(T_f-T_\infty)\varphi\left(\frac{x}{2\sqrt{\alpha_0t}}\right)+T_\infty&0<x<s(t),\,t>0\\
\label{Sol:s}&s(t)=2\lambda\sqrt{\alpha_0t}&t>0,
\end{align}
\end{subequations}
if and only if the function $\varphi$ and the parameter $\lambda>0$ satisfy the following differential problem:
%\vspace*{-0.75cm}
\begin{subequations}\label{VarphiProblem}
\begin{align}
\label{Eq:Varphi}&[(1+\beta y(\eta))y'(\eta)]'+2\eta y'(\eta)=0
&0<\eta<\lambda\\
\label{Cond:Varphi0}&y'(0)+\beta y(0) y'(0)-
\gamma y(0)=0&\\
\label{Cond:VarphiLambda}& y(\lambda)=1&
\end{align}
\end{subequations}
together with the following condition:
\begin{equation}\label{Eq:VarphiLambda}
\frac{\varphi'(\lambda)}{\lambda}=\frac{2}{(1+\beta)\text{Ste}_\infty},
\end{equation}
where $\text{Ste}_\infty=\dfrac{c(T_f-T_\infty)}{l}>0$ (Stefan number) and $\gamma=2\text{Bi}>0$ with $Bi=\dfrac{h_0\sqrt{\alpha_0}}{k_0}$ (generalized Biot number).
\end{theorem}   

Any function $\varphi$ that satisfies problem (\ref{VarphiProblem}) will be referred to as {\em Generalized Modified Error} (GME) function. In the following Section we will prove that such functions exist.

\section{Similarity solutions.}\label{Sect:EyU}
In this Section we will analyse the existence and uniqueness of the similarity solution (\ref{SimSol}) to problem (\ref{StefanProblem}). By virtue of Theorem \ref{Th:characterisation}, it can be done through the analysis of problem (\ref{VarphiProblem})-(\ref{Eq:VarphiLambda}). First, we will study the differential problem (\ref{VarphiProblem}) by assuming that $\lambda$ is a positive given number. As we shall see shortly, it can be proved that problem (\ref{VarphiProblem}) has a unique non-negative analytic solution $\varphi$ when $\beta$ assumes small positive values. Then, we will analyse the relation between the found solution $\varphi$ and the parameter $\lambda$ through the study of the equation (\ref{Eq:VarphiLambda}).

\subsection{Analysis of problem (\ref{VarphiProblem}).}
This Section is devoted to the GME function. First, we will present a result on its existence and uniqueness. Then, it will be proved that the GME function is increasing and concave, as the classical error function is. Finally, explicit approximations are proposed for the GME function and several plots are presented for different values of the parameters involved in the physical problem (\ref{StefanProblem}).

\subsubsection{Existence and uniqueness of the GME function}\label{Sect:GME}
The ideas that will be developed in the following are based on the recent article \cite{CeSaTa2017}, where it was proved the existence of the ME function introduced in \cite{ChSu1974}. Through a fixed point 
strategy, we will prove the existence of the GME function, $\varphi$. 

All throughout this Section we will consider $\lambda>0$, $\gamma>0$ and $\beta\geq 0$. We will denote with $X$ to the set of all bounded analytic functions $h:[0,\lambda]\to\mathbb{R}$. It is well known that $X$ is a Banach space with the supremum norm $||\cdot||_\infty\,$, which is defined by:
\begin{equation}\label{NormSup}
||h||_\infty=\sup\left\{|h(x)|\,:\,0\leq x\leq \lambda\right\}\quad (h\in X).
\end{equation}  
The subset of $X$ given by all non-negative functions which are bounded by 1 will be referred to as $K$, that is:
\begin{equation*}
K=\left\{h\in X\,:\,0\leq h\text{ and }||h||_\infty\leq 1\right\}.
\end{equation*}
Note that $K$ is a non-empty closed subset in $(X,||\cdot{}||_\infty)$. Finally, for each $h\in K$ we will write $\Psi_h=1+\beta h$. 
\begin{remark}\label{Re:PsiBounds}
$1\leq\Psi_h\leq 1+\beta$ for any $h\in K$.
\end{remark}

The main idea in the analysis below is to study the nonlinear problem (\ref{VarphiProblem}) through the linear problem given by:
\begin{subequations}\label{VarphiProblemStar}
\begin{align}
\label{Eq:VarphiLin}&[\Psi_h(\eta) y'(\eta)]'+2\eta y'(\eta)=0& 0<\eta<\lambda\\
&\Psi_h(0)y'(0)-\gamma y(0)=0&\\
&y(\lambda)=1,&
\end{align}
\end{subequations}
where $h$ is a known function which belongs to $K$. The advantage in considering problem (\ref{VarphiProblemStar}) lies in the fact that the differential equation (\ref{Eq:VarphiLin}) can be easily solved as a linear equation of first order in $y'$.  

\begin{lemma}\label{Le:SolLinear}
Let $h\in K$. The only solution $y$ to problem (\ref{VarphiProblemStar}) is given by:
\begin{equation}\label{y}
y(\eta)=D_h\left(\frac{1}{\gamma}+\displaystyle\bigintsss_0^\eta
\frac{\exp\left(-2\displaystyle\int_0^x\frac{\xi}{\Psi_h(\xi)}d\xi\right)}{\Psi_h(x)}dx\right)\quad 0<\eta<\lambda,
\end{equation}
with $D_h$ defined by:
\begin{equation}\label{Dh}
D_h=\gamma\left(1+\gamma\displaystyle\bigintsss_0^\lambda
\frac{\exp\left(-2\displaystyle\int_0^x\frac{\xi}{\Psi_h(\xi)}d\xi\right)}{\Psi_h(x)}dx\right)^{-1}.
\end{equation}
Moreover, $y\in K$.
\end{lemma}

\begin{proof}
By observing that the constant $D_h$ given by (\ref{Dh}) is well defined since $\Psi_h$ is never zero, the proof follows easily by checking that the function $y$ given by (\ref{y}) satisfies problem (\ref{VarphiProblemStar}).
\end{proof}

The next result is a direct consequence of the previous Lemma.

\begin{theorem}\label{Th:FixedPoint}
Let $y\in K$. Then $y$ is a solution to problem (\ref{VarphiProblem}) if and only if $y$ is a fixed point of the operator $\tau$ from $K$ to $X$ given by:
\begin{equation}\label{tau}
\left(\tau h\right)(\eta)=D_h\left(\frac{1}{\gamma}+\displaystyle\bigintsss_0^\eta
\frac{\exp\left(-2\displaystyle\int_0^x\frac{\xi}{\Psi_h(\xi)}d\xi\right)}{\Psi_h(x)}dx\right)\quad 0<\eta<\lambda,\quad (h\in K)
\end{equation}
where $D_h$ is defined by (\ref{Dh}).
\end{theorem}

In the following we will focus on analysing the existence of fixed points of $\tau$. 

\begin{theorem}\label{Th:tauK}
$\tau(K)\subset K$.
\end{theorem}
\begin{proof}
Let $h\in K$. We have that:
\begin{enumerate}
\item[i)] $\tau h$ is an analytic function, since $h\in X$. 
\item[ii)] $0\leq \tau h$, since $0<D_h$.
\item[iii)] $||\tau h||_\infty\leq 1$, since $\left|\left(\tau h\right)(\eta)\right|\leq \left(\tau h\right)(\lambda)=1$ for all $0<\eta<\lambda$.
\end{enumerate} 
Then, $\tau h\in K$.
\end{proof}

\begin{lemma}\label{Le:Cotas}
Let $h, h_1, h_2\in K$ and $\eta\in[0,\lambda]$. We have:
\begin{enumerate}
%\item[]
\item[a)] $\displaystyle\bigintsss_0^\eta\left|
\frac{\exp\left(-2\displaystyle\int_0^x\frac{\xi}{\Psi_{h_1}(\xi)}d\xi\right)}{\Psi_{h_1}(x)}-
\frac{\exp\left(-2\displaystyle\int_0^x\frac{\xi}{\Psi_{h_2}(\xi)}d\xi\right)}{\Psi_{h_2}(x)}\right|dx
$\\[0.25cm]$\leq\frac{\sqrt{\pi}}{4} \beta(1+\beta)^{1/2}(3+\beta)||h_1-h_2||_{\infty}$
%\item[]
\item[b)] $0<D_h\leq \gamma$. 
\end{enumerate}
\end{lemma}

\begin{proof}
\begin{enumerate}
\item[]
\item[a)] See \cite[Lemma 2.1]{CeSaTa2017}.
\item[b)] It is a direct consequence of the positivity of $\Psi_h$  (see Remark \ref{Re:PsiBounds}).
\end{enumerate}
\end{proof}

\begin{lemma}\label{Le:Eq}
Let $g$ be the real function defined by:
\begin{equation}\label{g}
g(x)=\frac{\sqrt{\pi}}{2}\gamma x(1+x)^{1/2}(3+x)\quad x>0.
\end{equation}
The equation:
\begin{equation}\label{eq:delta1}
g(x)=1 \quad x>0
\end{equation} 
has an only positive solution $\beta_1=\beta_1(\gamma)$. 
\end{lemma}

\begin{proof}
It follows from the fact that $g$ is an increasing function in $\mathbb{R}^+$ with $\displaystyle\lim_{x\to 0^+}g(x)=0$ and $\displaystyle\lim_{x\to +\infty}g(x)=+\infty$.
\end{proof}

\begin{theorem}\label{Th:Contraction}
Let $\beta_1$ the only positive solution to equation (\ref{eq:delta1}). If $0\leq\beta<\beta_1$, then $\tau$ is a contraction.
\end{theorem}

\begin{proof}
Let $h_1, h_2\in K$ and $\eta\in[0,\lambda]$. From Lemma \ref{Le:Cotas} and:
\begin{equation*}
\begin{split}
\left|(\tau h_1)(\eta)-(\tau h_2)(\eta)\right|&\leq 
\gamma\displaystyle\bigintsss_0^\eta\left|
\frac{\exp\left(-2\displaystyle\int_0^x\frac{\xi}{\Psi_{h_1}(\xi)}d\xi\right)}{\Psi_{h_1}(x)}-
\frac{\exp\left(-2\displaystyle\int_0^x\frac{\xi}{\Psi_{h_2}(\xi)}d\xi\right)}{\Psi_{h_2}(x)}\right|dx\\
&+\left|D_{h_1}-D_{h_2}\right|\left|\frac{1}{\gamma} +\bigintsss_0^\eta
\frac{\exp\left(-2\displaystyle\int_0^x\frac{\xi}{\Psi_{h_2}(\xi)}d\xi\right)}{\Psi_{h_2}(x)} dx \; \right| \\
& \leq 2 \gamma\displaystyle\bigintsss_0^\eta\left|
\frac{\exp\left(-2\displaystyle\int_0^x\frac{\xi}{\Psi_{h_1}(\xi)}d\xi\right)}{\Psi_{h_1}(x)}-
\frac{\exp\left(-2\displaystyle\int_0^x\frac{\xi}{\Psi_{h_2}(\xi)}d\xi\right)}{\Psi_{h_2}(x)}\right|dx,
\end{split}
\end{equation*}
we have that that $||\tau h_1-\tau h_2||_\infty\leq g(\beta)||h_1-h_2||_\infty$. Recalling that $g$ is an increasing function, it follows that $\tau$ is a contraction when $\beta<\beta_1$.
\end{proof}

We are in a position now to formulate our main result.

\begin{theorem}\label{Th:EyU-GME}
Let $\beta_1$ as in Theorem \ref{Th:Contraction}. If $0\leq\beta<\beta_1$ then problem (\ref{VarphiProblem}) has a unique non-negative analytic solution.
\end{theorem}

\begin{proof}
It a direct consequence of Theorems \ref{Th:FixedPoint}, \ref{Th:tauK}, \ref{Th:Contraction} and the Banach Fixed Point Theorem. 
\end{proof}

\subsubsection{Some properties of the GME function}\label{Sect:GMEProp}
This Section is devoted to prove that the GME function $\varphi$ found in Section \ref{Sect:GME} shares the following properties with the classical error function erf:
\begin{equation}\label{GMEProp}
\text{i) }0\leq \varphi(\eta)\leq 1,
\hspace{2cm}
\text{ii) }0<\varphi'(\eta),
\hspace{2cm}
\text{iii)}\varphi''(\eta)<0\hspace{1cm}\forall\,0<\eta<\lambda.
\end{equation}

In the following we will consider $\gamma>0$, $\lambda>0$, $\beta_1$ the solution to equation (\ref{eq:delta1}) and $0\leq\beta<\beta_1$. The first property in (\ref{GMEProp}) is an immediate consequence of the fact that $\varphi\in K$. In order to prove ii) in (\ref{GMEProp}), it will be enough to show that $\varphi'(\eta)\neq 0$ for all $0<\eta<\lambda$, since $\varphi(0)\leq 1$ and $\varphi(\lambda)=1$. Let us assume that there exists $0<\eta_0<\lambda$ such that $\varphi'(\eta_0)=0$. As we shall see shortly, we will reach a contradiction. Since equation (\ref{Eq:Varphi}) can be written as:
\begin{equation}\label{Eq:Varphi-bis}
(1+\beta y(\eta))y''(\eta)+\beta(y'(\eta))^2+2\eta y'(\eta)=0\quad 0<\eta<\lambda,
\end{equation}
and we have that:
\begin{equation}\label{IneqVarphi}
1+\beta \varphi(\eta_0)>0,
\end{equation}
we obtain that $\varphi''(\eta_0)=0$. From this, by differentiating (\ref{Eq:Varphi-bis}) and taking into account (\ref{IneqVarphi}), it follows that $\varphi'''(\eta_0)=0$. We continue in this fashion obtaining that $\varphi^{(n)}(\eta_0)=0$ for all $n\in\mathbb{N}$. But this implies that $\varphi\equiv 0$ in $[0,\lambda]$, since $\varphi$ is an analytic function. This contradicts $\varphi(\lambda)=1$. Finally, we have that the last property in (\ref{GMEProp}) is a direct consequence of i), ii) and the fact that $\varphi''$ is given by (see (\ref{Eq:Varphi-bis}) and (\ref{IneqVarphi})):
\begin{equation*}
\varphi''(\eta)=-\frac{\beta(\varphi'(\eta))^2+2\eta\varphi'(\eta)}{1+\beta\varphi(\eta)}\quad 0<\eta<\lambda.
\end{equation*}

%\begin{remark}
%Properties (\ref{GMEProp}) also hold for the ME function $\Phi$ introduced in \cite{ChSu1974}, as it was shown in \cite{CeSaTa2017-Manuscript}.
%\end{remark}

\subsubsection{Approximation of the GME function.}\label{Sect:GMEApprox}
The following is devoted to obtain explicit approximations for the GME function $\varphi$ found in Section \ref{Sect:GME}. 

Let $\lambda>0$, $\beta>0$, $\gamma>0$ be given. Based on the assumption that problem (\ref{VarphiProblem}) has a solution $\varphi$ that can be represented as:
\begin{equation}\label{VarphiSeries}
\varphi(\eta)=\displaystyle\sum_{n=0}^\infty\beta^n\varphi_n(\eta)\quad 0<\eta<\lambda,
\end{equation}
where $\varphi_n$ are real functions that must be determined, we will propose approximations $\varphi^{(N)}$ of the GME function given as:
\begin{equation}\label{VarphiSeries}
\varphi^{(N)}(\eta)=\displaystyle\sum_{n=0}^N\beta^n\varphi_n(\eta)\quad 0<\eta<\lambda,
\end{equation}
with $N\in\mathbb{N}_0$.

If $\varphi$ is given by (\ref{VarphiSeries}), equation (\ref{Eq:Varphi}) is formally equivalent to:
\begin{equation}
\displaystyle\sum_{n=1}^\infty\left(\displaystyle\sum_{k=1}^n a(\eta,k-1,n-k)+b(\eta,n)\right)\beta^n+b(\eta,0)=0\quad 0<\eta<\lambda,
\end{equation}
where:
\begin{subequations}
\begin{align*}
&a(\eta,n,m)=\varphi'_n(\eta)\varphi'_m(\eta)+
\varphi_n(\eta)\varphi''_m(\eta)&0\leq\eta\leq\lambda,\,n,\,m\in\mathbb{N}_0\\
&b(\eta,n)=\varphi''_n(\eta)+2\eta\varphi_n'(\eta)&0\leq\eta\leq\lambda,\,n\in\mathbb{N}_0
\end{align*}
\end{subequations}

Similarly, we have that (\ref{Cond:Varphi0}) is formally equivalent to:
\begin{equation}
\displaystyle\sum_{n=1}^\infty\left(\displaystyle\sum_{k=1}^na_0(k-1,n-k)+b_0(n)\right)+b_0(0)=0,
\end{equation}
where:
\begin{subequations}
\begin{align*}
&a_0(n,m)=\varphi_n'(0)\varphi_m(0)&n,\,m\in\mathbb{N}_0\\
&b_0(n)=\varphi_n'(0)-\gamma\varphi_n(0)&n\in\mathbb{N}_0
\end{align*}
\end{subequations}

Therefore, if the functions $\varphi_n$ are such that:
\begin{subequations}\label{VarphinPreliminar}
\begin{align}
&\displaystyle\sum_{k=1}^n a(\eta,k-1,n-k)+b(\eta,n)=0,\hspace{1cm} b(\eta,0)=0&0<\eta<\lambda,\,n\in\mathbb{N}\\
&\displaystyle\sum_{k=1}^na_0(k-1,n-k)+b_0(n)=0,\hspace{1.45cm} b_0(0)=0&n\in\mathbb{N}\\
&\varphi_0(\lambda)=1&\\
&\varphi_n(\lambda)=0&n\in\mathbb{N}
\end{align}
\end{subequations}
then the function $\varphi$ given by (\ref{VarphiSeries}) is a formal solution to (\ref{VarphiProblem}). Thus, functions $\varphi_n$ might be determined through problem (\ref{VarphinPreliminar}). Let us observe it states that $\varphi_0$ must be a solution to: 
\begin{subequations}\label{VarphiProblem-0}
\begin{align}
\label{Eq:Varphi0}&2\eta\varphi_0'(\eta)+\varphi_0''(\eta)=0\quad 0<\eta<\lambda\\
&\varphi_0'(0)-\gamma\varphi_0(0)=0\\
&\varphi_0(\lambda)=1
\end{align}
\end{subequations}
while each $\varphi_n$, $n\in\mathbb{N}$, must satisfy:
\begin{subequations}\label{VarphiProblem-n}
\begin{align}
&2\eta\varphi_n'(\eta)+\varphi_n''(\eta)=g_n(\eta)\quad 0<\eta<\lambda\\
&\displaystyle\sum_{k=1}^n\varphi_{k-1}'(0)\varphi_{n-k}(0)+\varphi_n'(0)-\gamma\varphi_n(0)=0\\
&\varphi(\lambda)=0
\end{align}
\end{subequations}
with:
\begin{equation}\label{RHSVarphiProblem-n}
g_n(\eta)=-\displaystyle\sum_{k=1}^n
\left(\varphi'_{k-1}(\eta)\varphi'_{n-k}(\eta)+\varphi_{k-1}(\eta)\varphi_{n-k}''(\eta)\right)\quad 0<\eta<\lambda.
\end{equation}

\begin{remark}\label{Re:Beta0}
Observe that problem (\ref{VarphiProblem-0}) coincides with (\ref{VarphiProblem}) when $\beta=0$.
\end{remark}

In the following we will only work with the zero and first order approximations $\varphi^{(0)}$, $\varphi^{(1)}$. From elementary results in ordinary differential equations (see, for example \cite{PoZa1995}), it can be obtained that the solution $\varphi_0$ to problem (\ref{VarphiProblem-0}) is the function given by:
\begin{equation}\label{Varphi0}
\varphi_0(\eta)=\frac{2}{2+\gamma\sqrt{\pi}\erf(\lambda)}+
\frac{\gamma\sqrt{\pi}}{2+\gamma\sqrt{\pi}\erf(\lambda)}\erf(\eta)\quad 0\leq \eta\leq \lambda.
\end{equation}
Having obtained $\varphi_0$, we can now compute $\varphi_1$ through the problem (\ref{VarphiProblem-n}) (see (\ref{RHSVarphiProblem-n})), and obtain \cite{PoZa1995}: 
\begin{equation}\label{varphi_1}
\begin{split}
\varphi_1(\eta)&=
B_1(2+\gamma\sqrt{\pi}\erf(\eta))+
B_2
+
\frac{\gamma}{\nu^2}\left(5\sqrt{\pi}-2\eta\exp(-\eta^2)
-\gamma\pi\exp(-2\eta^2)\right.\\
&-\left.
2\gamma\sqrt{\pi}\eta\erf(\eta)\exp(-\eta^2)+\frac{\gamma\pi}{2}\erf^2(\eta)+\gamma\sqrt{\pi}\erf(\eta)\exp(-\eta^2)\right)
\end{split}
\end{equation}
with:
\begin{subequations}
\begin{align*}
&B_1=-\frac{B_2}{\nu}+\frac{\gamma}{\nu^3}\left(
-5\sqrt{\pi}\erf(\lambda)+2\lambda\exp(-\lambda^2)+
\gamma\pi\exp(-2\lambda^2)\right.\\
&+
\left.2\gamma\sqrt{\pi}\lambda\erf(\lambda)\exp(-\lambda^2)
-\frac{\gamma\pi}{2}\erf^2(\lambda)-\gamma\sqrt{\pi}\exp(-\lambda^2)\erf(\lambda)
\right)\\
&B_2=\frac{1}{\nu^2}(12+2\gamma+\gamma^2\pi)\\
&\nu=2+\gamma\sqrt{\pi}\erf(\lambda).
\end{align*}
\end{subequations}

Therefore, we have that $\varphi^{(0)}=\varphi_0$ and $\varphi^{(1)}=\varphi_0+\beta\varphi_1$, with $\varphi_0$ and $\varphi_1$ given by (\ref{Varphi0}), (\ref{varphi_1}) respectively. In order to analyse the relation between each of these approximations and the GME function $\varphi$, we define the error $\mathcal{E}^{(N)}$ as:
\begin{equation}\label{e}
\mathcal{E}^{(N)}=\max\left\{\left|\varphi(\eta)-\varphi^{(N)}(\eta)\right|\,:\,0\leq\eta\leq\lambda\right\}\quad (N=0,1).
\end{equation}

In the following, we will show that the GME function $\varphi$ converges uniformly to the zero order approximation $\varphi^{(0)}$, when $\beta\to 0$. As we shall see below, this is closely related to how the problem (\ref{VarphiProblem}) depends on the parameter $\beta$.

\begin{definition}\label{Def:Lips}
Let $0<b<\beta_1$. We will say that problem (\ref{VarphiProblem}) is Lipschitz continuous on the parameter $\beta$ over the interval $[0,b]$ if there exists $L>0$ such that for any $b_1, b_2 \in [0,b]$ the following inequality holds:
\begin{equation}
||\varphi_{b_1}-\varphi_{b_2}||_\infty\leq L|b_1-b_2|,
\end{equation}
where $\varphi_{b_1}$, $\varphi_{b_2}$ are the only solutions in $K$ to problem (\ref{VarphiProblem}) with parameters $b_1$, $b_2$, respectively.   
\end{definition}

It follows from Definition \ref{Def:Lips} that if problem (\ref{VarphiProblem}) is Lipschitz continuous on $\beta$ over some interval $[0,b]$ with $0<b<\beta_1$, then the GME function $\varphi$ converges uniformly on $0<\eta<\lambda$ to the function $\varphi_0$ given by (\ref{Varphi0}), when $\beta\to 0$ (see Remark \ref{Re:Beta0}). In other words, that $\mathcal{E}^{(0)} \to 0$ when $\beta\to 0$ and therefore,  that $\varphi_0$ is a good approximation of the GME function $\varphi$ when the positive parameter $\beta$ is small enough. In the following we will prove that problem (\ref{VarphiProblem}) is in fact Lipschitz continuous on the parameter $\beta$ over $[0,b]$ for any choice of $0<b<\beta_1$. 

\begin{lemma}\label{Le:Cota}
Let $h\in K$, $\eta\in[0,\lambda]$ and $b_1, b_2\in[0,\beta_1)$. We have:
\begin{equation*}
\displaystyle \int_0^\eta \left| \frac{\exp\left(-2\displaystyle\int_0^x\frac{\xi}{1+b_1 \varphi_h(\xi)}d\xi\right)}{1+b_1 \varphi_h(x)}-\frac{\exp\left(-2\displaystyle\int_0^x\frac{\xi}{1+b_2 \varphi_h(\xi)}d\xi\right)}{1+b_2 \varphi_h(\xi)} \right|  dx\leq
\frac{1}{2\gamma \beta_1} | b_1 -b_2 |.
\end{equation*}
\end{lemma}

\begin{proof}
Let $f$ be the real function defined on $\mathbb{R}^+_0$ by $f(x)=\exp(-2x)$ and:
\begin{equation*}
x_1= \int_0^x\frac{\xi}{1+b_1 h(\xi)}d\xi, \quad x_2= \int_0^x\frac{\xi}{1+b_2 h(\xi)}d\xi\quad (x>0 \text{ fixed}).
\end{equation*}

It follows from the Mean Value Theorem applied to function $f$ that:
\begin{equation*}
\left| f(x_1)-f(x_2)\right|=|f'(u)| |x_1-x_2|,
\end{equation*}
where $u$ is a real number between $x_1$ and $x_2$. Without any lost of generality, we will assume that $b_2\leq b_1$. Then, $x_1\leq x_2$. We have:
\begin{align*}
&|f'(u)|\leq |f'(x_1)|\leq  2 \exp \left( -\frac{x^2}{1+\beta_1}\right)\text{ since }||h||_\infty\leq 1,\\
& |x_1-x_2| \leq \frac{x^2}{2} |b_1 - b_2|.
\end{align*} 

Therefore, 
\begin{equation*}
\left|f(x_1)-f(x_2)\right| \leq  x^2  \exp \left( -\frac{x^2}{1+\beta_1}\right)|b_1 - b_2|
\end{equation*}

Then, we have:
\begin{equation*}
\begin{split}
\left|
\frac{f(x_1)}{1+b_1 h(x)}-\frac{f(x_2)}{1+b_2 h(x)}
\right|&=
\left|
\frac{f(x_1)-f(x_2)}{1+b_1 h(x)}+\frac{f(x_2)h(x)(b_2-b_1)}{(1+b_1h(x))(1+b_2h(x))}\right|\\
&\leq |f(x_1)-f(x_2)|+|f(x_2)||b_1-b_2|\\
&\leq \exp\left(-\frac{x^2}{1+\beta_1}\right)(x^2+1)|b_1 - b_2|.
\end{split}
\end{equation*}
The final bound is now obtained by integrating the last expression and by the definition of $\beta_1$.
\end{proof}

\begin{theorem}\label{theoLip}
Let $0<b<\beta_1$. The problem (\ref{VarphiProblem}) is Lipschitz continuous on the parameter $\beta$ over the interval $[0,b]$.
\end{theorem}

\begin{proof}
Let $\eta\in[0,\lambda]$, $b_1, b_2 \in [0,b]$ and $\varphi_{b_1}, \varphi_{b_2}\in K$ the solutions to problem (\ref{VarphiProblem}) with parameters $b_1$, $b_2$, respectively. 

Taking into consideration that $\varphi_{b_1}$ and $ \varphi_{b_2}$  are the fixed points of the operator $\tau_{b_1}$ and $\tau_{b_2}$ defined by (\ref{tau}) for $\beta=b_1$ and $\beta=b_2$, respectively, we have that:
\begin{equation*}
|\varphi_{b_1}(\eta)- \varphi_{b_2}(\eta)| \leq 
2 D_{\varphi_{b_1}} \int_0^\lambda \left| \frac{\exp\left(-2\displaystyle\int_0^x\frac{\xi}{1+b_1 \varphi_{b_1}(\xi)}d\xi\right)}{1+b_1 \varphi_{b_1}(x)}-\frac{\exp\left(-2\displaystyle\int_0^x\frac{\xi}{1+b_2 \varphi_{b_2}(\xi)}d\xi\right)}{1+b_2 \varphi_{b_2}(x)} \right| dx.
\end{equation*}

Now, from Lemmas \ref{Le:Cotas}, \ref{Le:Cota} it follows that:
\begin{equation*}
\left|\varphi_{b_1}(\eta)-\varphi_{b_2}(\eta)\right| \leq g(b_1)||\varphi_{b_1}-\varphi_{b_2}||_\infty+ \frac{1}{\beta_1} | b_1 -b_2 |.
\end{equation*}

Since $g$ is an increasing function, $g(\beta_1)=1$ and $b_1\leq b<\beta_1$, we have that $g(b_1)\leq g(b)<1$. Then:
\begin{equation*}
||\varphi_{b_1}-\varphi_{b_2}||_\infty\leq L|b_1-b_2|\quad \text{with }L=\frac{1}{ \beta_1 (1-g(b))}>0.
\end{equation*}
\end{proof}

We end this Section by presenting some comparisons between the GME function $\varphi$, and its zero and first order approximations $\varphi^{(0)}$, $\varphi^{(1)}$. Figures \ref{Fig:GME-GMEAprox-Gamma01-Beta1} to \ref{Fig:GME-GMEAprox-Gamma100-Beta1} show the evolution of the error $\mathcal{E}^{(N)}$, and the plots of the GME function $\varphi$ against the best approximation obtained between $\varphi^{(0)}$ and $\varphi^{(1)}$. Each Figure correspond to one value of $\gamma=0.1,1,10,100$. The GME function $\varphi$ was obtained after solving problem (\ref{VarphiProblem}) through the \texttt{bvodes} routine implemented in Scilab. Numerical computations were made by considering $\lambda\in[0,10]$ and a uniform mesh of step $0.01$ for the interval $[0,\lambda]$. For each choice of the parameter $\gamma$, equation (\ref{eq:delta1}) was numerically solved. The approximative solutions $\beta_1^*$ for each value of $\gamma$ are presented in Table \ref{Ta:Beta1}. According to the election of $\gamma$, parameter $\
beta_1$ was set as $\beta_1^*$ as Table \ref{Ta:Beta1} states. From Figures \ref{Fig:GME-GMEAprox-Gamma01-Beta1} to \ref{Fig:GME-GMEAprox-Gamma100-Beta1} it can be seen that good agreement between the GME function $\varphi$ and either the zero or first order approximations $\varphi^{(0)}$, $\varphi^{(1)}$ can be obtained. They also suggest that the election between $\varphi^{(0)}$ and $\varphi^{(1)}$ it is mediated by the value of $\gamma$.

\begin{table}[H]
\caption{{\em {\small Approximative solutions $\beta^*_1$ to the equation (\ref{eq:delta1}) for $\gamma=0.1,1,10,100$.}}}
\centering
\begin{tabular}{|l|c|c|c|c|}
\hline
$\gamma$&$0.1$&$1$&$10$&$100$\\
\hline
$\beta^*_1$&$1.55$&$3\times 10^{-1}$&$3.65\times 10^{-2}$&$3.75\times 10^{-3}$\\
\hline
\end{tabular}
\label{Ta:Beta1}
\end{table}

\begin{figure}[H]
\caption{{\em{\small Comparison between the GME function and its approximations for $\gamma=0.1$, $\lambda=10$ and $\beta=\beta_1^*$ (see Table \ref{Ta:Beta1}).}}}
\vspace*{-0.5cm}
\centering
\begin{subfigure}{0.45\textwidth}
\includegraphics[scale=0.28]{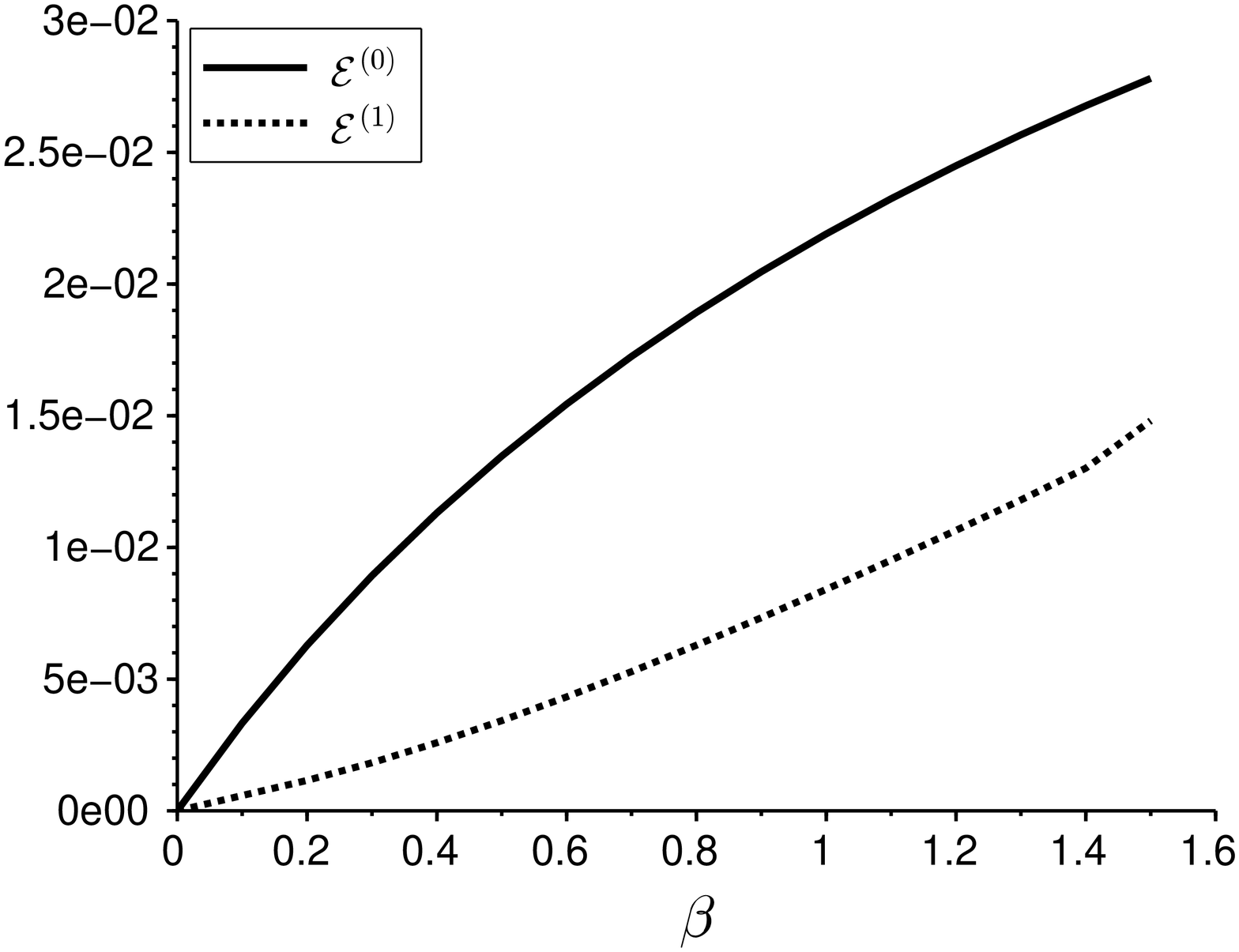}
\caption{{\small Error $\mathcal{E}^{(N)}$, $N=0,1$}}
\end{subfigure}
\begin{subfigure}{0.45\textwidth}
\includegraphics[scale=0.28]{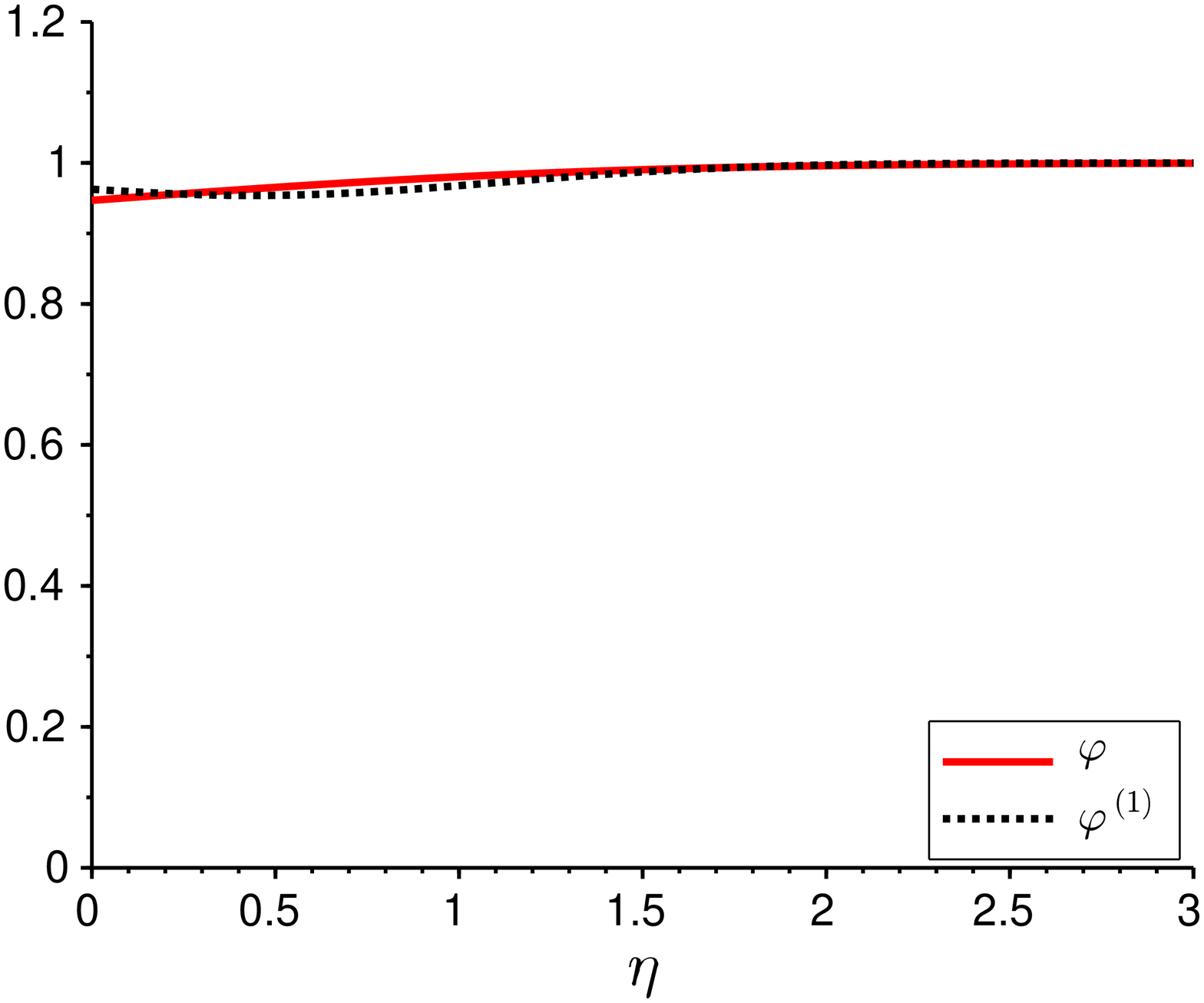}
\caption{{\small GME function $\varphi$ and the approximation $\varphi^{(1)}$}}
\end{subfigure}
\label{Fig:GME-GMEAprox-Gamma01-Beta1}
\end{figure}
\vspace*{-0.75cm}
\begin{figure}[H]
\caption{{\em{\small Comparison between the GME function and its approximations for $\gamma=1$, $\lambda=10$ and $\beta=\beta_1^*$ (see Table \ref{Ta:Beta1}).}}}
\vspace*{-0.5cm}
\centering
\begin{subfigure}{0.45\textwidth}
\includegraphics[scale=0.28]{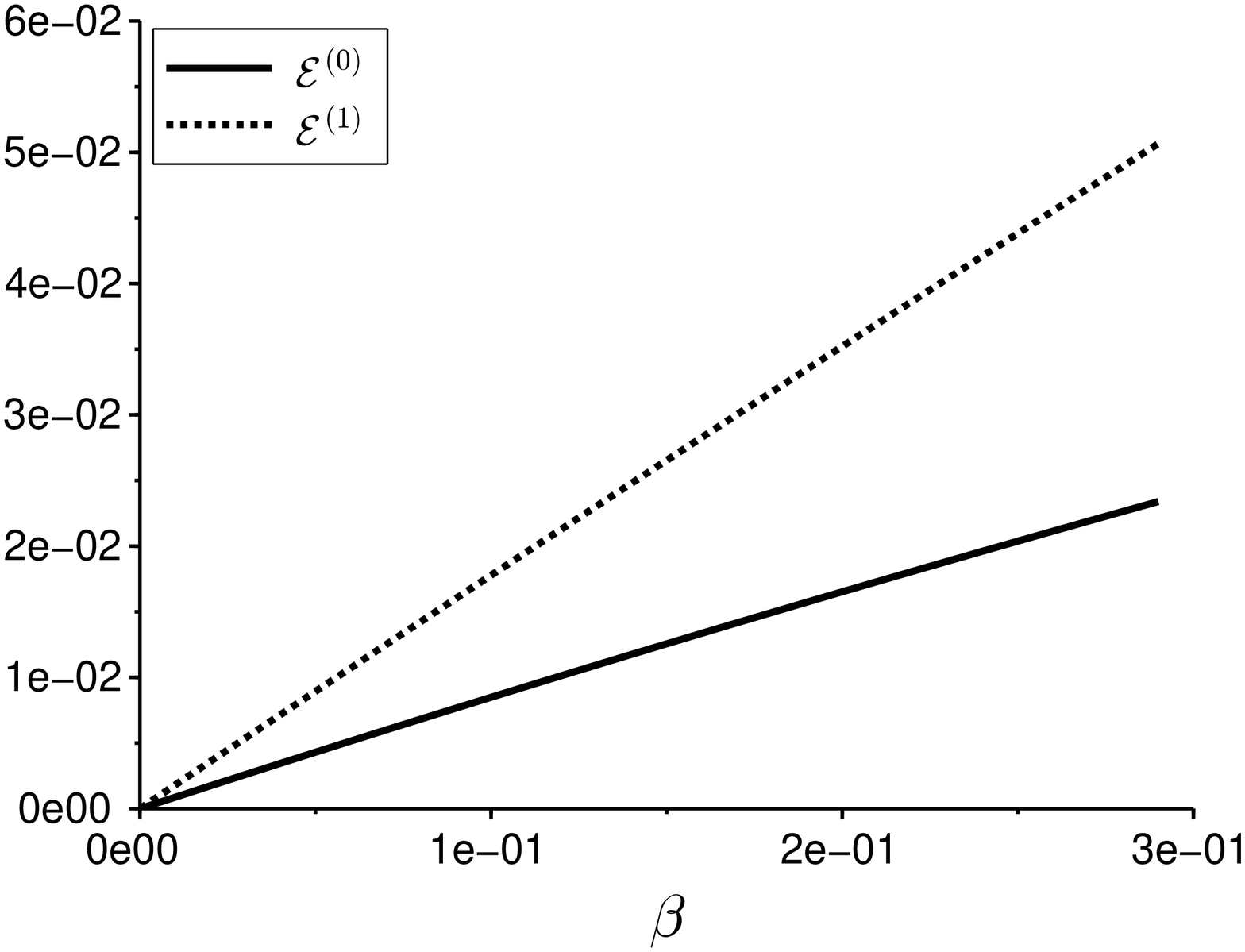}
\caption{{\small Error $\mathcal{E}^{(N)}$, $N=0,1$}}
\end{subfigure}
\begin{subfigure}{0.45\textwidth}
\includegraphics[scale=0.28]{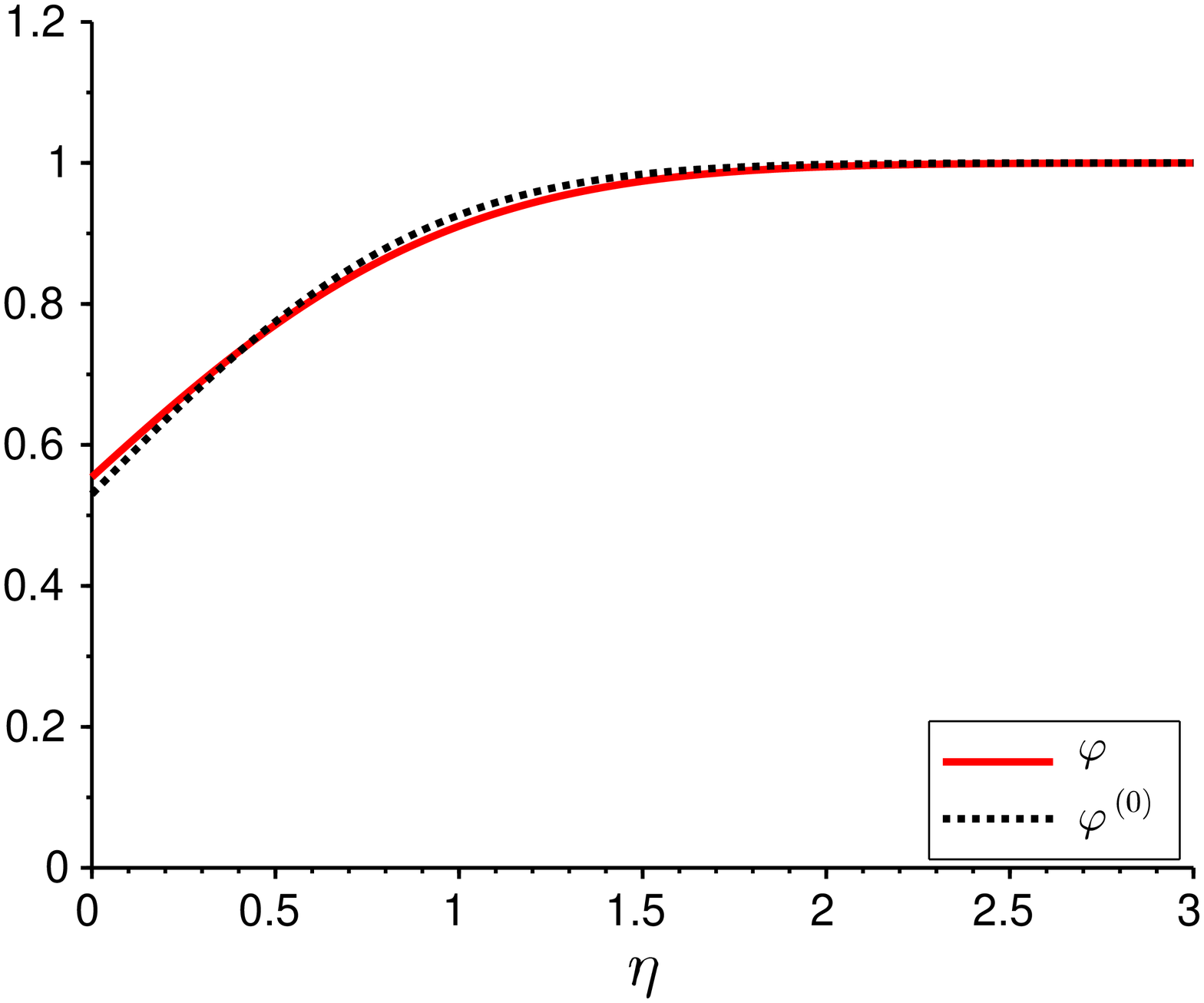}
\caption{{\small GME function $\varphi$ and the approximation $\varphi^{(0)}$}}
\end{subfigure}
\label{Fig:GME-GMEAprox-Gamma1-Beta1}
\end{figure}

\begin{figure}[H]
\caption{{\em{\small Comparison between the GME function and its approximations for $\gamma=10$, $\lambda=10$ and $\beta=\beta_1^*$ (see Table \ref{Ta:Beta1}).}}}
\vspace*{-0.5cm}
\centering
\begin{subfigure}{0.45\textwidth}
\includegraphics[scale=0.28]{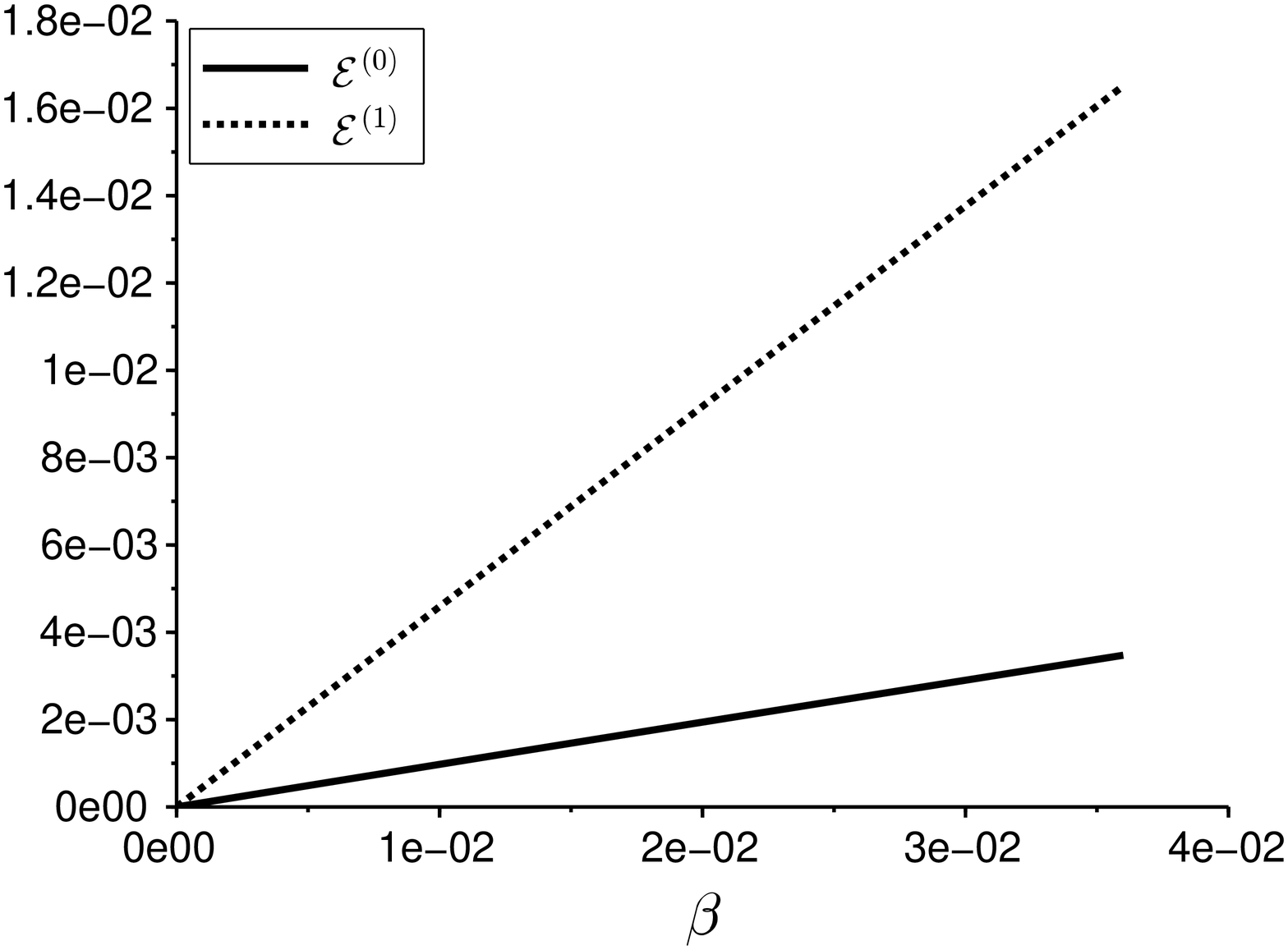}
\caption{{\small Error $\mathcal{E}^{(N)}$, $N=0,1$}}
\end{subfigure}
\begin{subfigure}{0.45\textwidth}
\includegraphics[scale=0.28]{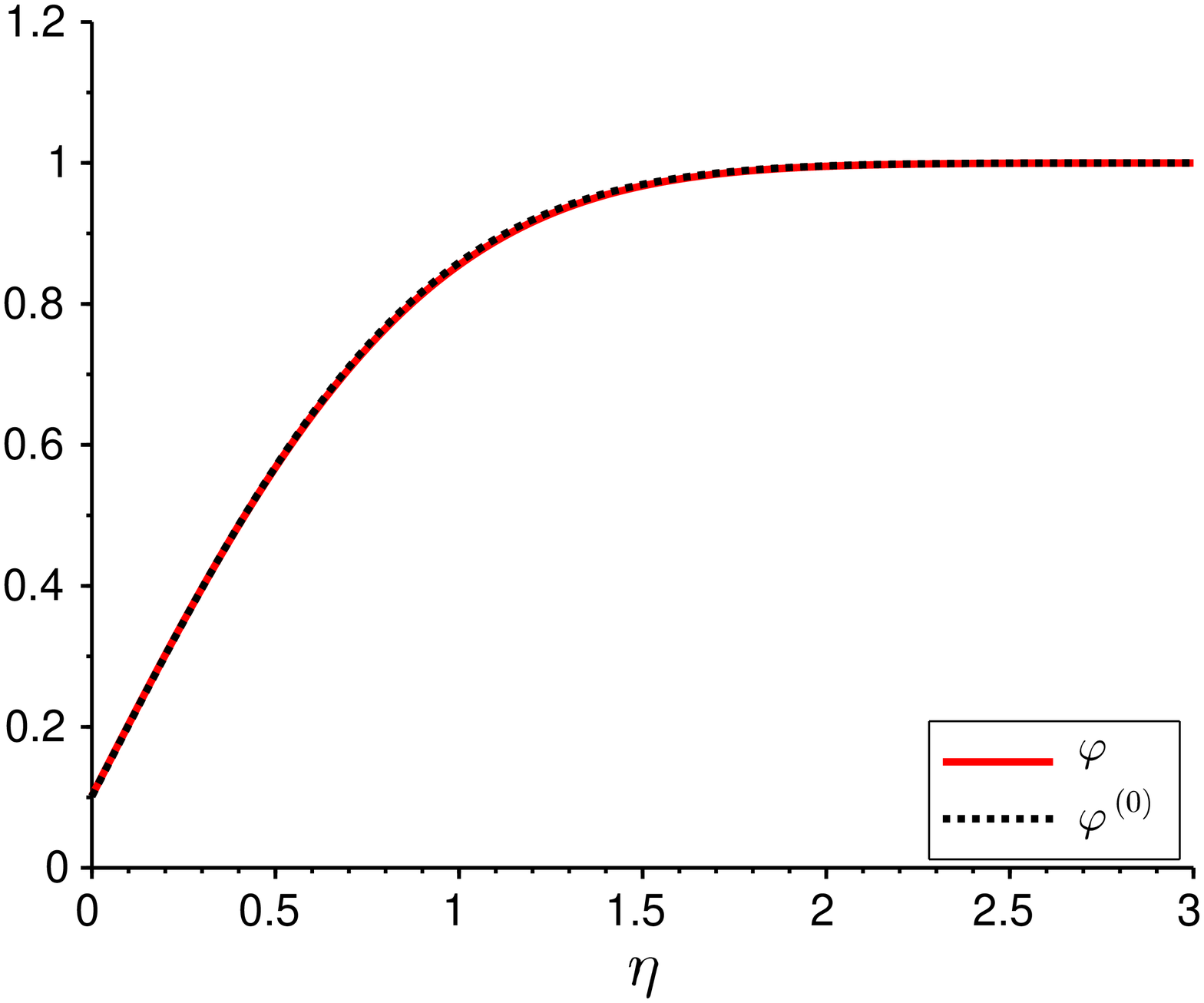}
\caption{{\small GME function $\varphi$ and the approximation $\varphi^{(0)}$}}
\end{subfigure}
\label{Fig:GME-GMEAprox-Gamma10-Beta1}
\end{figure}
\vspace*{-0.75cm}
\begin{figure}[H]
\caption{{\em{\small Comparison between the GME function and its approximations for $\gamma=100$, $\lambda=10$ and $\beta=\beta_1^*$ (see Table \ref{Ta:Beta1}).}}}
\vspace*{-0.5cm}
\centering
\begin{subfigure}{0.45\textwidth}
\includegraphics[scale=0.28]{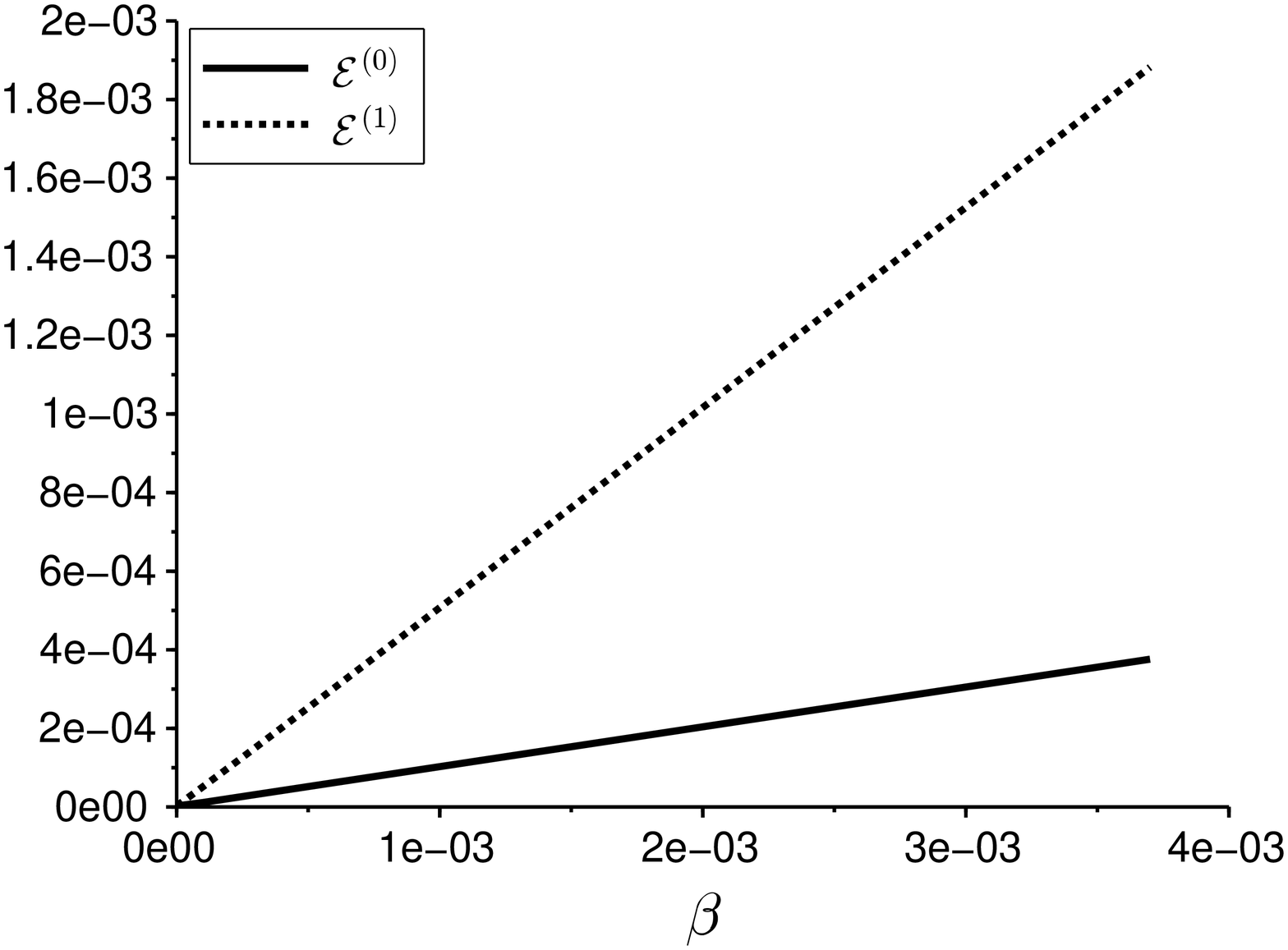}
\caption{{\small Error $\mathcal{E}^{(N)}$, $N=0,1$}}
\end{subfigure}
\begin{subfigure}{0.45\textwidth}
\includegraphics[scale=0.28]{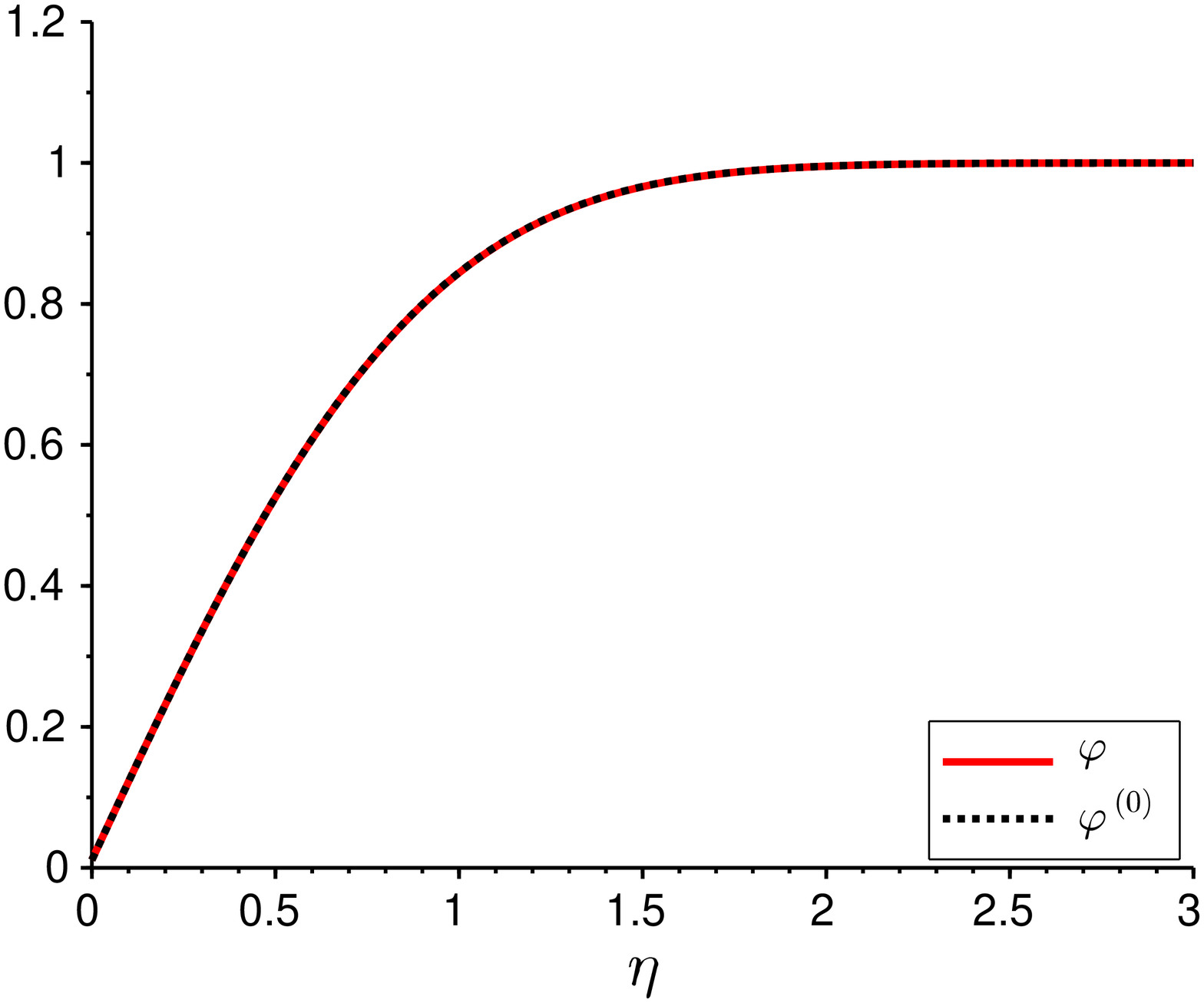}
\caption{{\small GME function $\varphi$ and the approximation $\varphi^{(0)}$}}
\end{subfigure}
\label{Fig:GME-GMEAprox-Gamma100-Beta1}
\end{figure}

\subsection{Analysis of equation (\ref{Eq:VarphiLambda}).}\label{Sect:Eq}
We will now investigate the relation between the solution $\varphi$ to problem (\ref{VarphiProblem}) found in Section \ref{Sect:GME} and the upper bound of its domain, that is the parameter $\lambda>0$, in order to analyse the existence of solution to equation (\ref{Eq:VarphiLambda}). Throughout this Section, $\beta_1$ will refer to the only positive solution to equation (\ref{eq:delta1}). 

\begin{lemma}\label{Le:GMEDerProp}
Let $\gamma>0$, $\lambda>0$, $0\leq\beta<\beta_1$. If $\varphi$ is the GME function which belongs to $K$, then we have:
\begin{equation*}
\text{i) } g_1(\lambda)<\varphi'(\lambda)<g_2(\lambda),
\hspace{3cm}
\text{ii) } \varphi'(\lambda)\leq \frac{\gamma}{1+\beta},
\end{equation*}
where $g_1(\lambda)$, $g_2(\lambda)$ are given by:
\begin{align*}
&g_1(\lambda)=\frac{\gamma}{1+\beta}\exp\left(-\lambda^2\right)\left(1+\gamma\displaystyle\int_0^\lambda\exp\left(-\frac{\eta^2}{1+\beta}\right)d\eta\right)^{-1},\\[0.25cm]
&g_2(\lambda)=\frac{\gamma}{1+\beta}\exp\left(-\frac{\lambda}{1+\beta}\right).
\end{align*}
\end{lemma}

\begin{proof}
\begin{enumerate}
\item[]
\item[i)] From the definition of $\varphi$ as the unique fixed point of the operator $\tau$ given by (\ref{tau}), we have that:
\begin{equation*}
\varphi'(\lambda)=\frac{D_h}{\Psi_\varphi(\lambda)}\exp\left(-2\displaystyle\int_0^\lambda\frac{\xi}{\Psi_\varphi(\xi)}d\xi\right).
\end{equation*}
Now the proof follows from:
\begin{equation*}
\gamma\left(1+\gamma\displaystyle\int_0^\lambda\exp\left(-\frac{\eta^2}{1+\beta}\right)d\eta\right)^{-1}
\leq D_h \leq \gamma,
\end{equation*}
the bounds for $\Psi_\varphi$ given in Remark \ref{Re:PsiBounds} and elementary boundedness techniques.
\item[ii)] It is a direct consequence of the second inequality in i).
\end{enumerate}
\end{proof}

\begin{remark}\label{Re:Lim}
The first part of Lemma \ref{Le:GMEDerProp} together with the Squeeze Theorem implies that $\displaystyle\lim_{\lambda\to 0^+}\varphi'(\lambda)=\frac{\gamma}{1+\beta}$.
\end{remark}

\begin{lemma}\label{Le:GMEContDep}
Let $\gamma>0$, $0\leq	\beta<\beta_1$. If $\varphi$ is the GME error function which belongs to $K$, then $\varphi'$ is continuous on the parameter $\lambda>0$.
\end{lemma}

\begin{proof}
After the change of variables:
\begin{equation*}
y(\eta)=z\left(\zeta\right),\hspace{2cm}\zeta=\frac{\eta}{\lambda},
\end{equation*}
we have that problem (\ref{VarphiProblem}) is equivalent to:
\begin{align*}
&\left[\left(1+\beta z(\zeta)\right)z'(\zeta)\right]'+2\lambda^2\zeta z(\zeta)=0\quad\quad 0<\zeta<1\\
&\left[1+\beta z(0)\right]z'(0)-\lambda\gamma z(0)=0\\
&z(1)=1.
\end{align*}
It follows from the fixed point expression of $\varphi$ that $\varphi'$ belongs to $K$. Applying Theorem 7.5 of \cite{CoLe1987} to the previous differential problem in the space $(0,1) \times (K \times K) \times \mathbb{R}^+$, we have that its solution is $C^1$ on the parameter $\lambda >0$.
\end{proof}

\begin{theorem}\label{Th:H}
Let $\gamma>0$. If $0\leq\beta<\beta_1$ then equation (\ref{Eq:VarphiLambda}) has at least one solution.  
\end{theorem}

\begin{proof}
For any $\lambda>0$, let $\varphi$ be the only solution in $K$ to problem (\ref{VarphiProblem}) on the domain $[0,\lambda]$. Let also be $H$ the real function defined by:
\begin{equation}\label{H}
H(\lambda)=\frac{\varphi'(\lambda)}{\lambda}\quad \lambda>0.
\end{equation}

Since $H$ is a continuous function (see Lemma \ref{Le:GMEContDep}) that satisfies (see Lemma \ref{Le:GMEDerProp} and Remark \ref{Re:Lim}):
\begin{equation}\label{H0+Infty}
\displaystyle\lim_{\lambda\to 0^+}H(\lambda)=+\infty,
\hspace*{2cm}
\displaystyle\lim_{\lambda\to +\infty}H(\lambda)=0,
\end{equation} 
the theorem follows by recalling that the RHS of equation (\ref{Eq:VarphiLambda}) is a positive number.
\end{proof}

In Figure \ref{Fig:H} we present some plots for the function $H$ defined by (\ref{H}). To compute $H$, is was considered the same numerical framework described at the end of Section \ref{Sect:GMEApprox} with $\lambda\in[0,5]$. Plots in Figure \ref{Fig:H} suggest that the solution found in Theorem \ref{Th:H} is the unique solution to equation (\ref{Eq:VarphiLambda}).

\begin{figure}[H]
\caption{{\em{\small Function $H$ for $\gamma=0.1, 1, 10, 100$, $\lambda\in[0,5]$ and  $\beta=\beta^*_1\simeq\beta_1$ (see Table \ref{Ta:Beta1}).}}}
\vspace*{-0.5cm}
\centering
\begin{subfigure}{0.45\textwidth}
\includegraphics[scale=0.28]{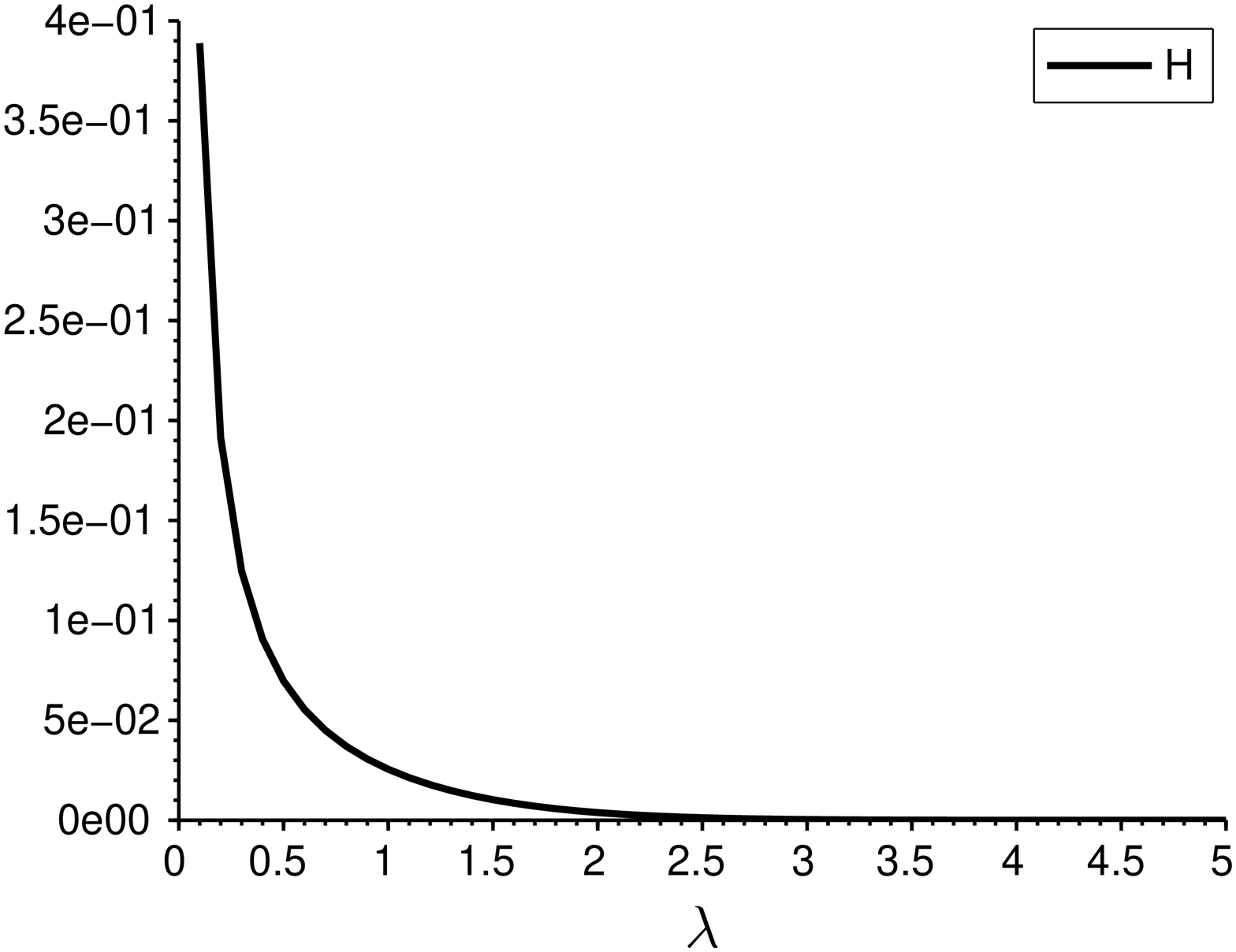}
\caption{{\small $\gamma=0.1$}}
\end{subfigure}
\begin{subfigure}{0.45\textwidth}
\includegraphics[scale=0.28]{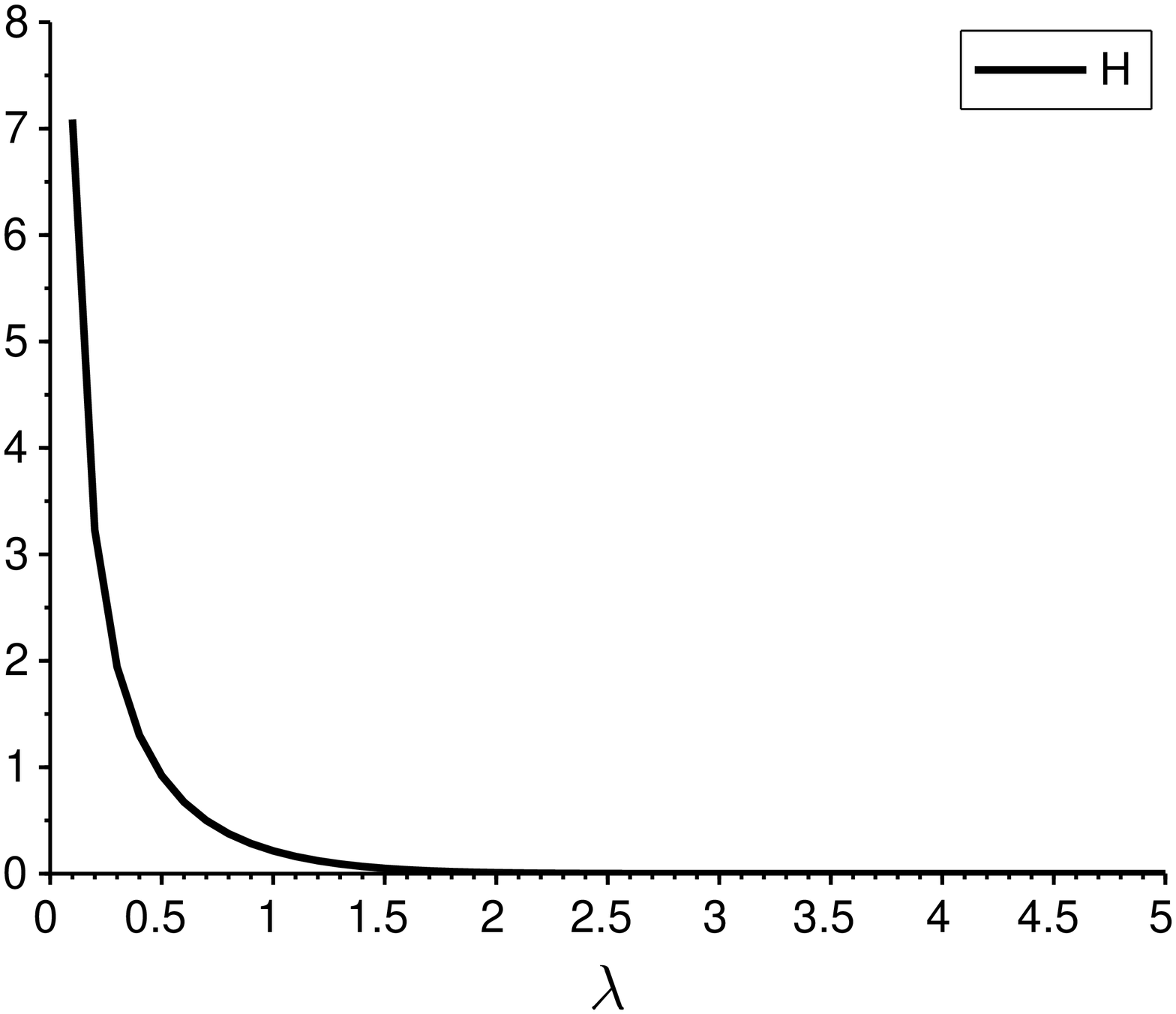}
\caption{{\small $\gamma=1$}}
\end{subfigure}
\begin{subfigure}{0.45\textwidth}
\includegraphics[scale=0.28]{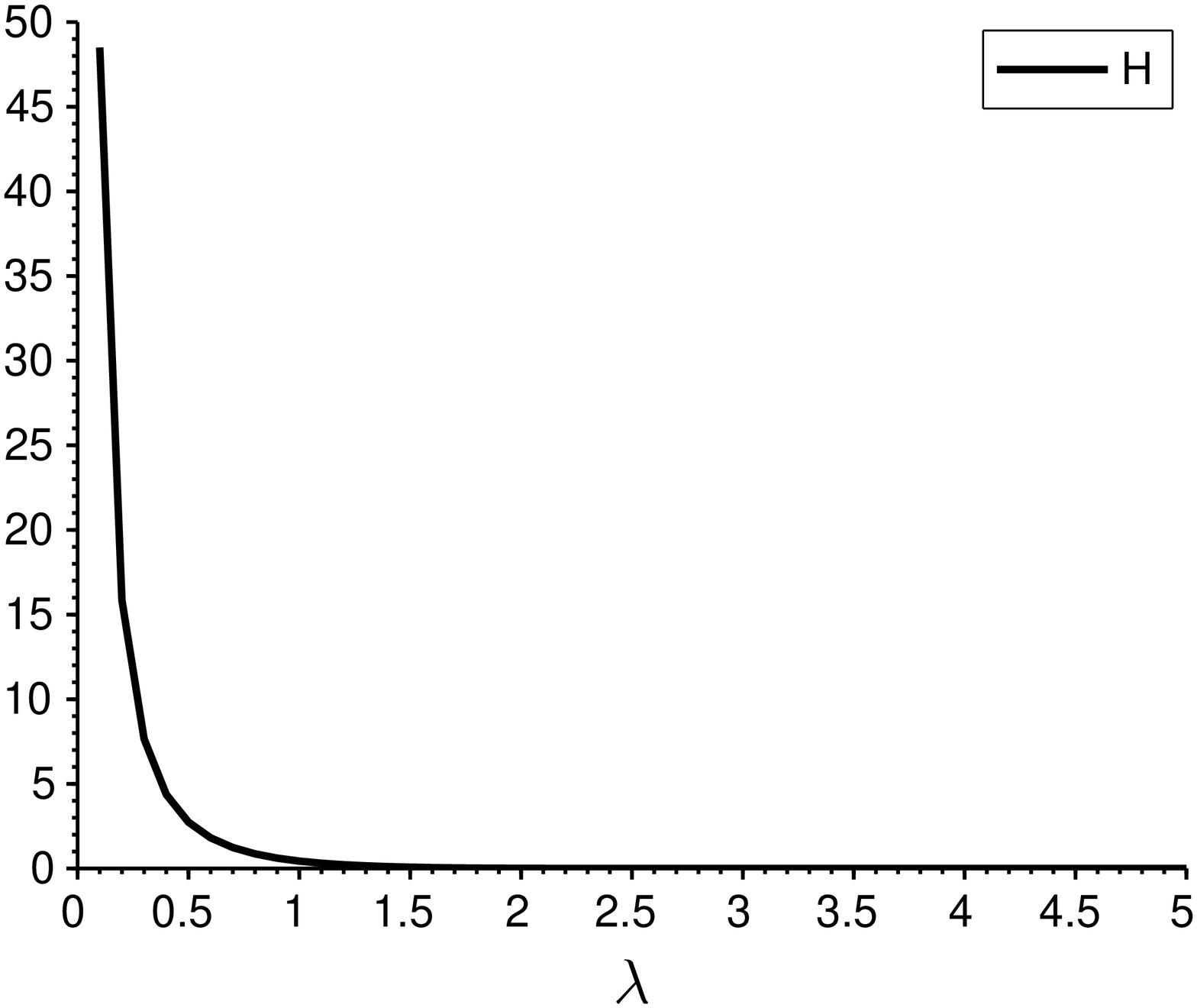}
\caption{{\small $\gamma=10$}}
\end{subfigure}
\begin{subfigure}{0.45\textwidth}
\includegraphics[scale=0.28]{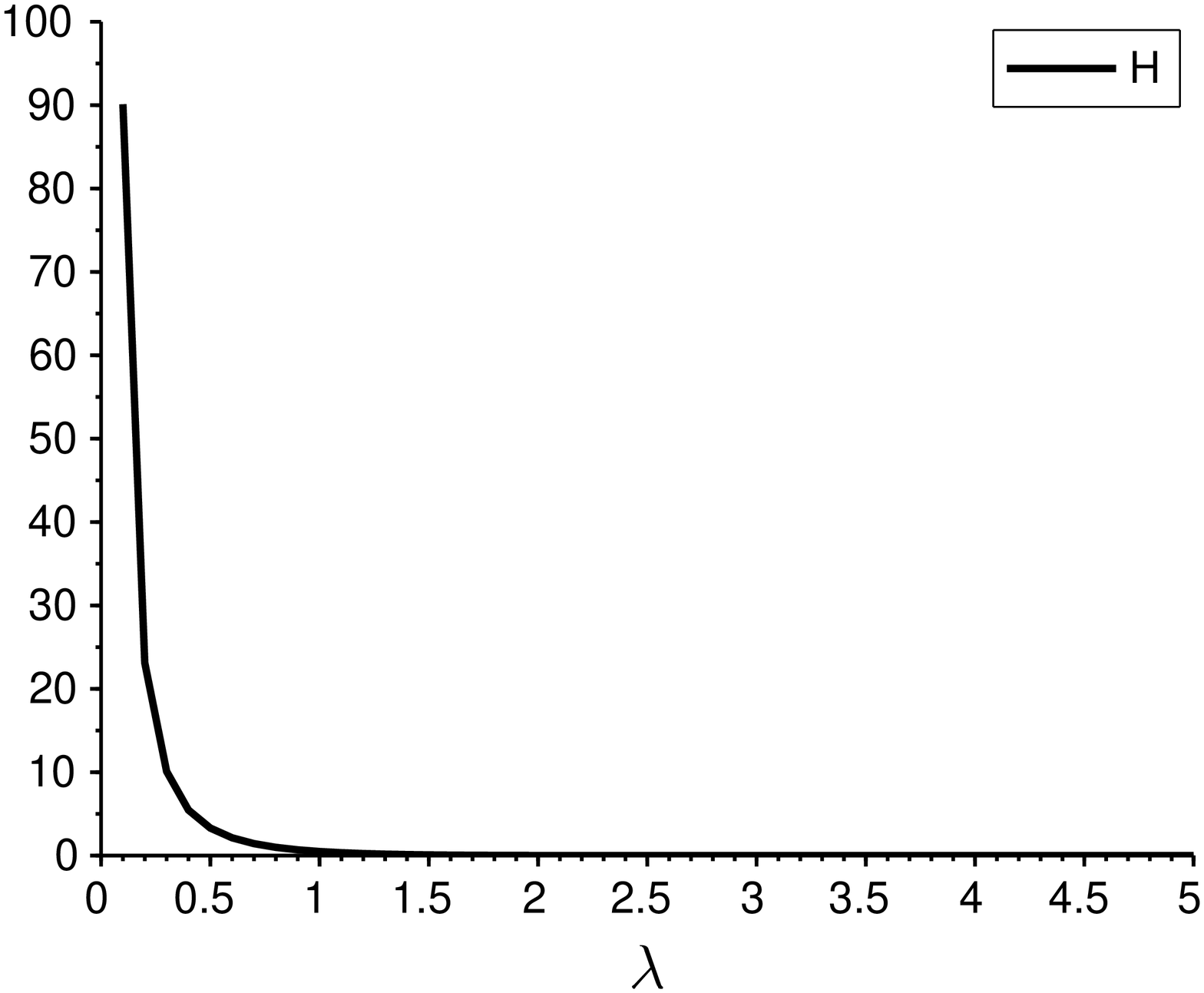}
\caption{{\small $\gamma=100$}}
\end{subfigure}
\label{Fig:H}
\end{figure}

\section{Relation with other Stefan problems}\label{Sect:Comparison}
\subsection{Relation with the Stefan problem with constant thermal conductivity}
Let $\lambda>0$, $\gamma=2\text{Bi}>0$ be given. As it was already noted in Section \ref{Sect:GMEApprox}, the solution to problem (\ref{VarphiProblem}) when $\beta=0$ (constant thermal conductivity) is the function $\varphi_0$ given by (\ref{Varphi0}). From this, we have that condition (\ref{Eq:VarphiLambda}) can be written as:
\vspace{-0.2cm}
\begin{equation}\label{Eq:lambda0}
\lambda\exp(\lambda^2)\left(1+\text{Bi}\erf(\lambda)\right)=\text{Ste}_\infty\text{Bi}.
\end{equation}

Therefore, the Stefan problem (\ref{StefanProblem}) with $\beta=0$ has the similarity solution $T$, $s$ given by:
\begin{equation}\label{Sol:T0}
T(x,t)=\frac{1+\text{Bi}\sqrt{\pi}\erf\left(\frac{x}{2\sqrt{\alpha_0t}}\right)}{1+\text{Bi}\sqrt{\pi}\erf(\lambda)}(T_f-T_\infty)+T_\infty\quad 0<x<s(t),\,t>0
\end{equation}
and (\ref{Sol:s}) if and only if $\lambda$ satisfies (\ref{Eq:lambda0}). This result has been already obtained in \cite{Ta2017}, where it was studied the phase-change process considered here but with constant thermal conductivity.

\subsection{Relation with the Stefan problem with Dirichlet condition}
Let us consider now the Stefan problem (\ref{StefanProblem}) with the Dirichlet boundary condition (\ref{Cond:Dirichlet}) in place of the convective one given by (\ref{Cond:Conv}). We will refer to it as problem (\ref{StefanProblem}$^\dag$). 

By following the same steps that led us to Theorem \ref{Th:characterisation}, we obtain that problem (\ref{StefanProblem}$^\dag$) has the similarity solution $T$, $s$ given by (\ref{SimSol}) if and only if the function $\varphi$ and the parameter $\lambda$ satisfy condition (\ref{Eq:VarphiLambda}) and the differential problem given by equation (\ref{Eq:Varphi}), condition (\ref{Cond:VarphiLambda}) and:
\vspace*{-0.2cm}
\begin{equation}\label{Cond:Varphi0-Dirichlet}
y(0)=0.\tag{\ref{Cond:Varphi0}$^\dag$}
\end{equation}
We will refer to the function $\varphi$ and the parameter $\lambda$ associated to problem (\ref{StefanProblem}$^\dag$) as $\varphi^\dag$ and $\lambda^\dag$, respectively. Problem (\ref{StefanProblem}$^\dag$) was studied in \cite{ChSu1974}, where it was obtained almost the same similarity solution than the one presented above. In \cite{ChSu1974}, the function $\varphi^\dag$ is defined over $\mathbb{R}_0^+$ through equation (\ref{Eq:Varphi}), condition (\ref{Cond:Varphi0-Dirichlet}) and  $y(+\infty)=1$ (that is, as the ME function $\Phi$). This last change adds some extra conditions that must be satisfied by the temperature distribution $T$. But it is avoidable, as we are showing here. Let (\ref{VarphiProblem}$^\dag$) be the problem given by (\ref{Eq:Varphi}), (\ref{Cond:Varphi0-Dirichlet}), (\ref{Cond:VarphiLambda}). In \cite{CeSaTa2017} it was proved the existence and uniqueness of the ME function $\Phi$ for small positive values of $\beta$ through a fixed point strategy. By performing the same analysis for problem (\ref{VarphiProblem}$^\dag$), we obtain that it has a unique non-negative bounded analytic solution $\varphi^\dag$ for any given $\lambda^\dag>0$. Moreover, we have that $\varphi^\dag$ is the unique fixed point of the operator $\tau^\dag$ from $K$ to itself defined by:
\begin{equation}\label{tauDag}
\left(\tau^\dag h\right)(\eta)=C_h\displaystyle\bigintsss_0^\eta
\frac{\exp\left(-2\displaystyle\int_0^x\frac{\xi}{\Psi_h(\xi)}d\xi\right)}{\Psi_h(x)}dx\quad 0<\eta<\lambda,\quad (h\in K)
\end{equation}
where $C_h$ is given by:
\begin{equation}
C_h=\left(\displaystyle\bigintsss_0^\lambda
\frac{\exp\left(-2\displaystyle\int_0^x\frac{\xi}{\Psi_h(\xi)}d\xi\right)}{\Psi_h(x)}dx\right)^{-1}.
\end{equation}

From the definitons of $\tau$ and $\tau^\dag$ given in (\ref{tau}) and (\ref{tauDag}), it follows that $\tau h\to\tau^\dag h$ (pointwise) when $\gamma\to+\infty$ for any function $h\in K$. Then, $\varphi\to\varphi^\dag$ (pointwise) when $\gamma\to+\infty$. When we consider $\gamma=2 Bi$ (see Theorem \ref{Th:Contraction}), $\gamma\to+\infty$ is equivalent to $h_0\to+\infty$. Thus, the solution to problem (\ref{StefanProblem}$^\dag$) can be obtained as the limit case of the solution to problem (\ref{StefanProblem}) when the coefficient $h_0$ that characterizes the heat transfer coefficient at $x=0$ goes to infinity. This agrees well with the physical interpretation of temperature and convective boundary conditions (see Remark \ref{Re:Conv-Dirichlet}, \cite{AlSo1993,CaJa1959}).

We end this Section with some plots for the GME function $\varphi$. From Figure \ref{Fig:GME-ME} it can be seen that it converges pointwise to the solution $\varphi^\dag$ to problem (\ref{VarphiProblem}$^\dag$). By an abuse of notation, we have also referred to $\varphi^\dag$ as ME function. The plots for both GME and ME functions were obtained after solving problems (\ref{VarphiProblem}) and (\ref{VarphiProblem}$^\dag$) for  $\lambda=10$ in the same numerical framework described at the end of Section \ref{Sect:GMEApprox}. Although it was considered $\lambda=10$, functions were drawn over the interval $[0,3]$ aiming at a better visualization.   

\begin{figure}[H]
\caption{{\em{\small GME and ME functions for $\gamma=0.1, 1, 10, 100$, $\lambda=10$ and $\beta=\beta^*_1\simeq\beta_1$ (see Table \ref{Ta:Beta1}).}}}
\vspace*{-0.5cm}
\centering
\begin{subfigure}{0.45\textwidth}
\includegraphics[scale=0.28]{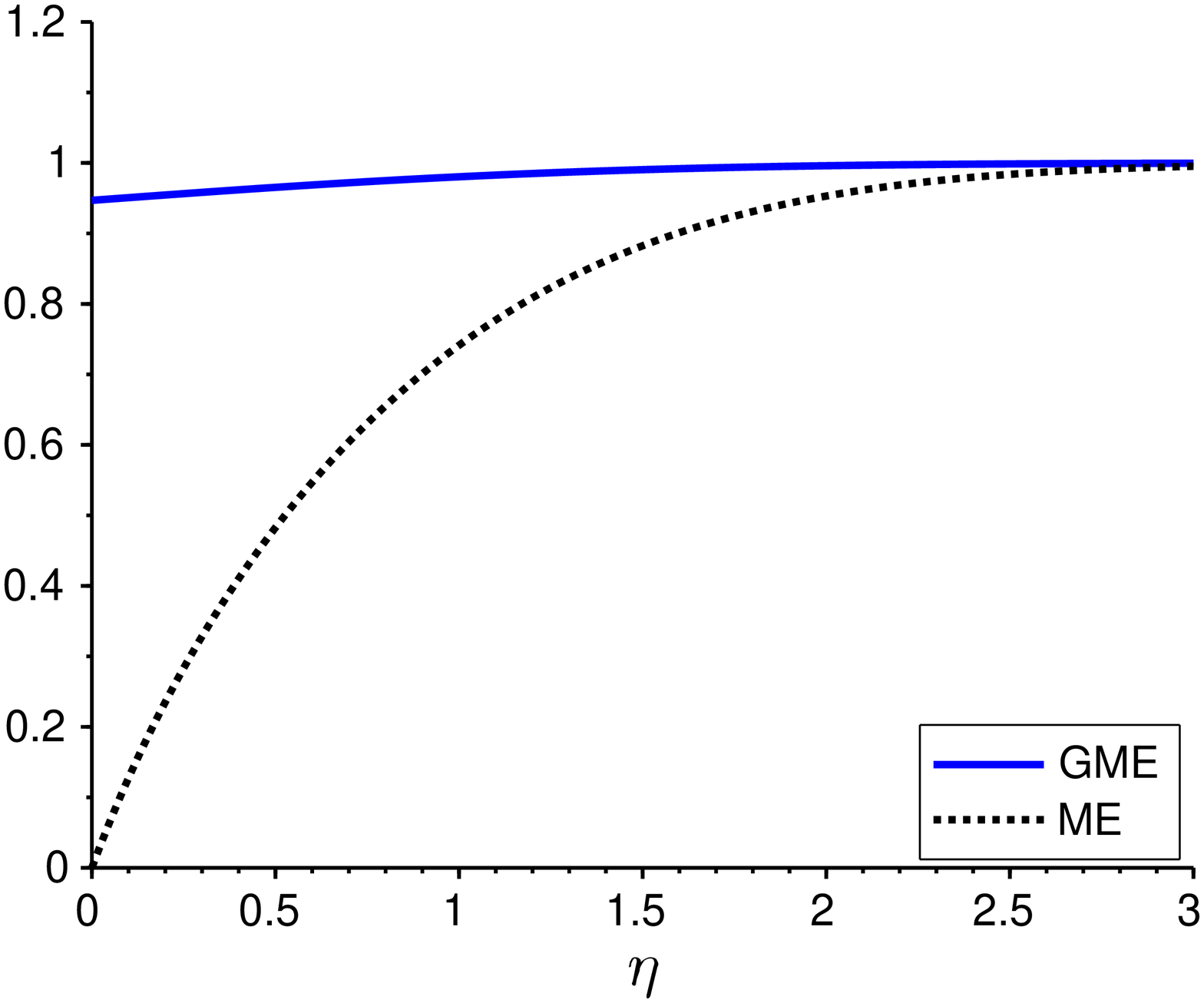}
\caption{{\small $\gamma=0.1$}}
\end{subfigure}
\hspace*{1cm}
\begin{subfigure}{0.45\textwidth}
\includegraphics[scale=0.28]{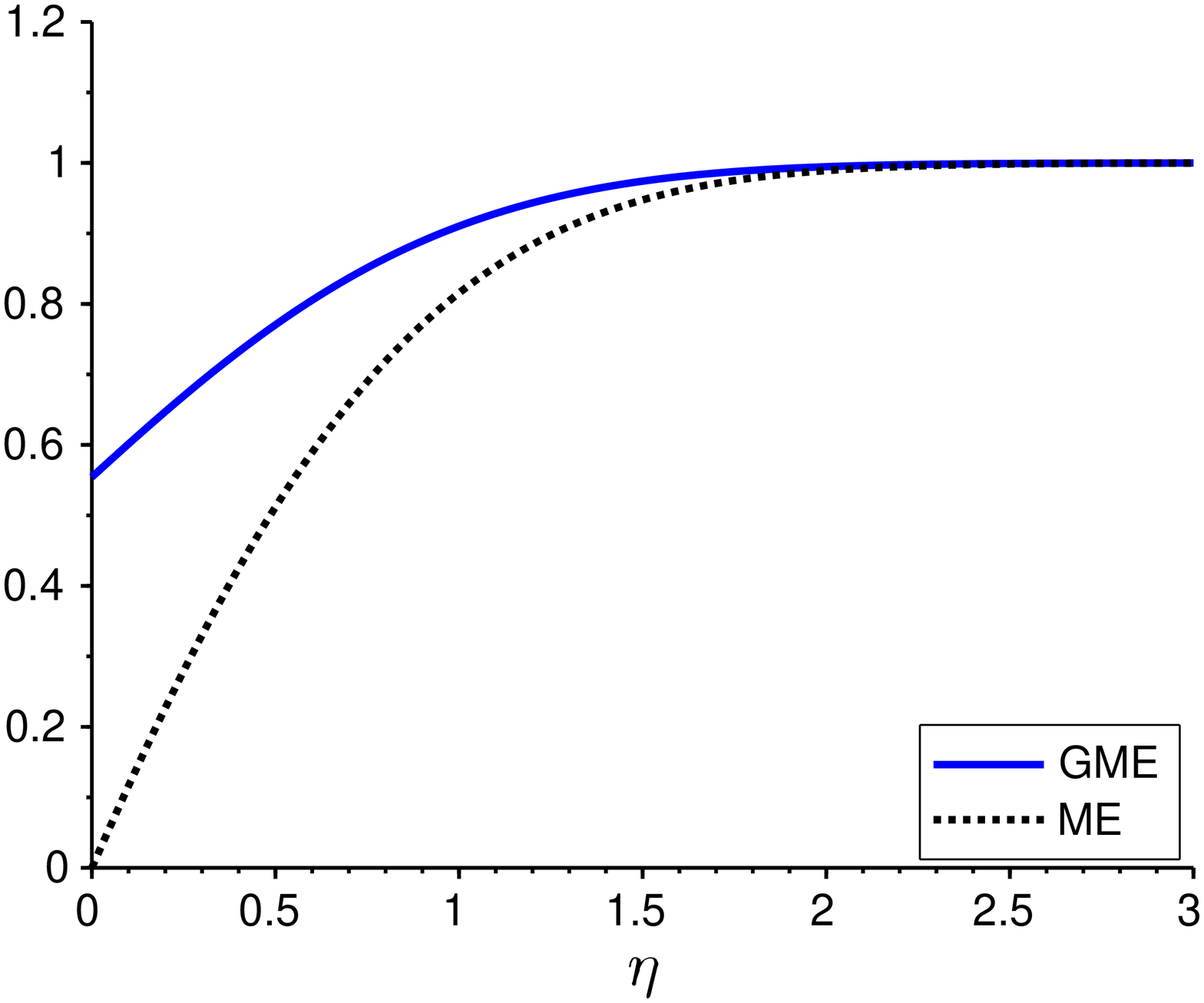}
\caption{{\small $\gamma=1$}}
\end{subfigure}\\[-0.25cm]
\begin{subfigure}{0.45\textwidth}
\includegraphics[scale=0.28]{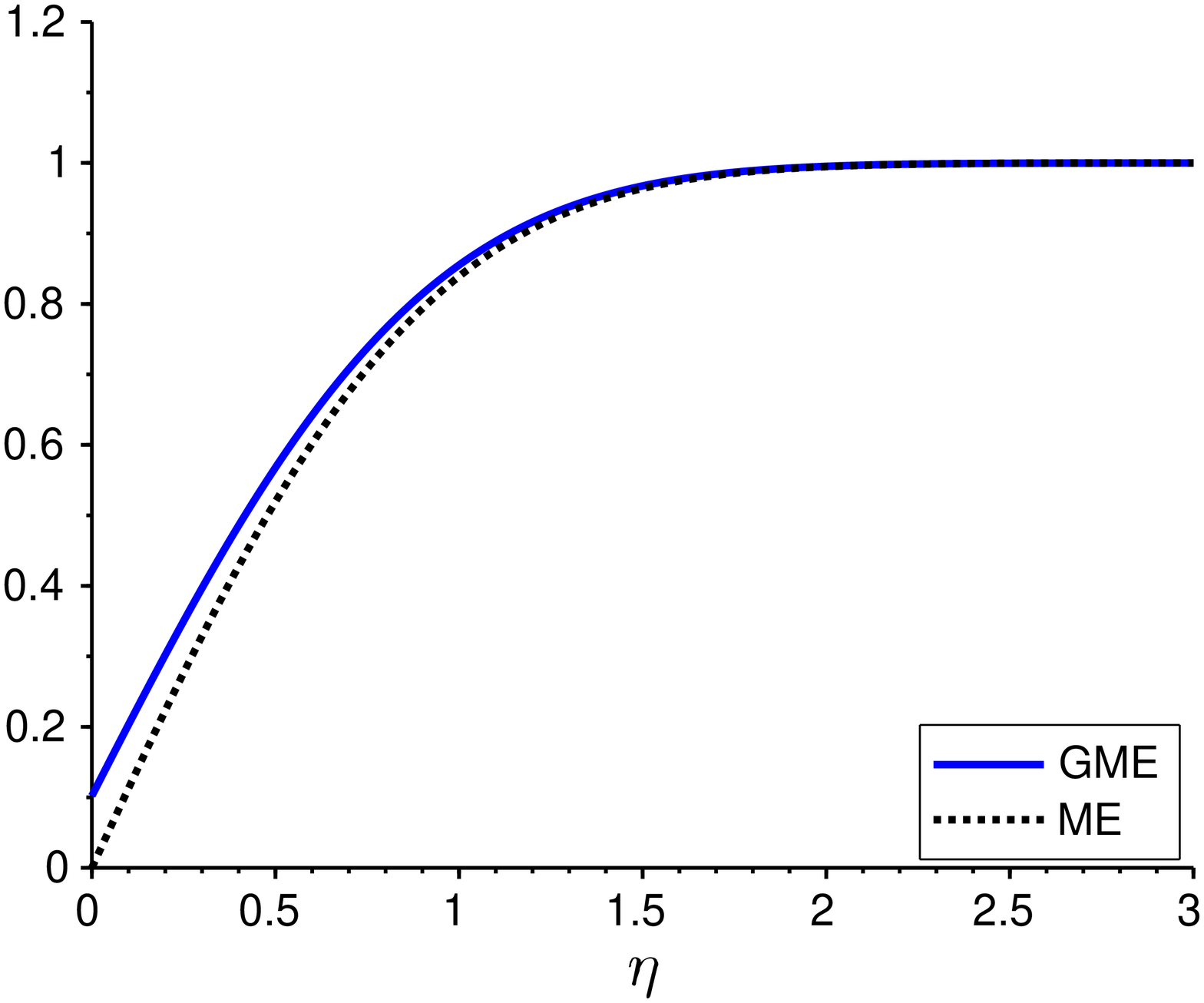}
\caption{{\small $\gamma=10$}}
\end{subfigure}
\hspace*{1cm}
\begin{subfigure}{0.45\textwidth}
\includegraphics[scale=0.28]{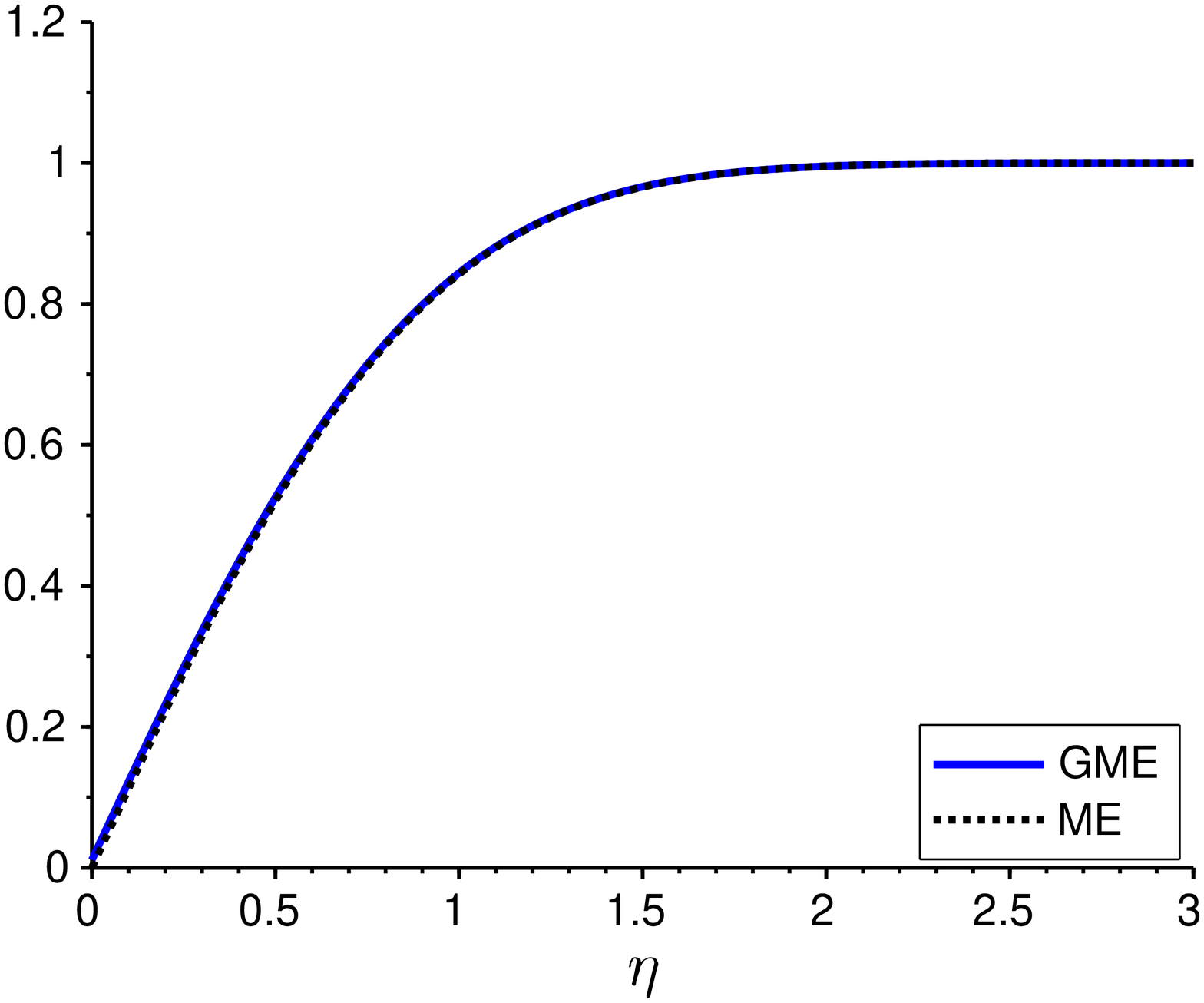}
\caption{{\small $\gamma=100$}}
\end{subfigure}
\label{Fig:GME-ME}
\end{figure}

\section{Conclusions}
In this article we have presented an exact solution of similarity type for a one-phase Stefan problem with temperature-dependent thermal conductivity and a Robin boundary condition. The temperature distribution was defined through a {\em Generalized Modified Error} (GME) function. This was defined as the solution to a nonlinear boundary problem of second order, for which it was proved a result on existence and uniqueness of solutions. From this, the existence of similarity solutions was proved. It was also shown that results for the Stefan problems with either constant thermal conductivity or Dirichlet boundary conditions can be obtained as particular o limit cases of the results presented in this article. 

Since the GME function is only available from numerical computations, it was proposed an strategy to obtain explicit approximations for it. Several values from the parameters involved in the physical problem were considered in the analysis of errors between the GME function and the two proposed approximations. The analysis performed suggest that the choice of the best approximation between those presented here depends on the values of the parameters. Nevertheless, good agreement can be obtained with both of them. From these explicit approximations, those for the temperature distribution can be obtained since it linearly depends on the GME function. In order to give some properties of the temperature distribution, it were also investigated some properties of the GME function. It was proved that it is a non-negative bounded analytic function which is increasing and concave, just as the classical error function is.   

\subsection*{Acknowledgments}
This paper has been partially sponsored by the Project PIP No. 0275 from CONICET-UA (Rosario, Argentina) and AFOSR-SOARD Grant FA 9550-14-1-0122.

%-------------------------------------------------------------------------------------
%\nocite{*}
\bibliographystyle{plain}
\bibliography{References_2017-06-15}
%-------------------------------------------------------------------------------------

\end{document}